\documentclass[journal]{ieeetran}
\usepackage{cite}
\usepackage{graphicx}
\usepackage{epstopdf}
\usepackage{amsmath}
\usepackage{cases}
\usepackage[caption=false,font=footnotesize]{subfig}
\usepackage{fixltx2e}
\usepackage{amssymb}
\usepackage{mathtools}
\usepackage{bm}
\usepackage{multirow}
\usepackage{color}
\usepackage{algorithm}
\usepackage{algorithmic}
\usepackage{amsthm}
\usepackage{mathrsfs}
\usepackage{textcomp,mathcomp}
\usepackage{threeparttable}
\usepackage{breakurl}
\newtheorem{theorem}{Theorem}

\newcommand{\tabincell}[2]{\renewcommand\arraystretch{0.9}\begin{tabular}{@{}#1@{}}#2\end{tabular}}

\begin{document}

\title{Planning of Off-Grid Renewable Power to Ammonia Systems with Heterogeneous Flexibility: A Multistakeholder Equilibrium Perspective }

\author{
Yangjun~Zeng,~\IEEEmembership{Student Member,~IEEE},
Yiwei~Qiu,~\IEEEmembership{Member,~IEEE},
Jie~Zhu,~\IEEEmembership{Student Member,~IEEE},
Shi~Chen,~\IEEEmembership{Member,~IEEE},
Tianlei~Zang,~\IEEEmembership{Member,~IEEE},
Buxiang~Zhou,~\IEEEmembership{Member,~IEEE},
Ge~He, and
Xu~Ji

\thanks{Financial support was obtained from the National Key R\&D Program of China (2021YFB4000503) and the National Natural Science Foundation of China (52377116, 52377115). \textit{(Corresponding author: Yiwei Qiu)}}
\thanks{Y. Zeng, Y. Qiu, J. Zhu, S. Chen, T. Zang, and B. Zhou are with the College of Electrical Engineering, Sichuan University, Chengdu 610065, China.
(e-mail: ywqiu@scu.edu.cn)
}%
\thanks{G. He and X. Ji are with the School of Chemical Engineering, Sichuan University, Chengdu 610065, China.}
}
\maketitle
\begin{abstract}
  Off-grid renewable power to ammonia (ReP2A) systems present a promising pathway toward carbon neutrality in both the energy and chemical industries. However, due to chemical safety requirements, the limited flexibility of ammonia synthesis poses a challenge when attempting to align with the variable hydrogen flow produced from renewable power. This necessitates the optimal sizing of equipment capacity for effective and coordinated production across the system. Additionally, an ReP2A system may involve multiple stakeholders with varying degrees of operational flexibility, complicating the planning problem. This paper first examines the multistakeholder sizing equilibrium (MSSE) of the ReP2A system. First, we propose an MSSE model that accounts for individual planning decisions and the competing economic interests of the stakeholders of power generation, hydrogen production, and ammonia synthesis. We then construct an equivalent optimization problem based on Karush--Kuhn--Tucker (KKT) conditions to determine the equilibrium.
  Following this, we decompose the problem in the temporal dimension and solve it via multicut generalized Benders decomposition (GBD) to address long-term balancing issues.
  Case studies based on a realistic project reveal that the equilibrium does not naturally balance the interests of all stakeholders due to their heterogeneous characteristics. Our findings suggest that benefit transfer or re-arrangement ensure mutual benefits and the successful implementation of ReP2A projects.
\end{abstract}

\begin{IEEEkeywords}
  renewable power to ammonia, sizing, investment equilibrium, multistakeholder interests, green hydrogen % multi-type buffer systems
\end{IEEEkeywords}

\section*{Nomenclature}
{\subsection{Abbreviations}
\begin{IEEEdescription}[\IEEEusemathlabelsep\IEEEsetlabelwidth{superscript}]
\addcontentsline{toc}{section}{Nomenclature}
\item[ReP2A] Renewable power to ammonia
\item[RG, rg] Renewable power generation stakeholder
\item[HP, hp] Hydrogen production stakeholder
\item[AS, as] Ammonia synthesis stakeholder
\item[WT, wt] Wind turbine
\item[PV, pv] Photovoltaic
%\item[ES] Energy storage
\item[BES, bes] Battery energy storage
\item[VC, vc] Var compensation
\item[AE, ae] Alkaline electrolyzer
\item[HST, hst] Hydrogen storage
\item[ASY, asy] Ammonia synthesis
\item[AST, ast] Ammonia storage
%\item[SOC] State of charge
\end{IEEEdescription}}
\vspace{-5pt}
\subsection{Indices}
\begin{IEEEdescription}[\IEEEusemathlabelsep\IEEEsetlabelwidth{superscript}]
\addcontentsline{toc}{section}{Nomenclature}
%\item[$t$] Index for time period
%\item[$n$] Index for ELZs
\item[$t$] Index for time intervals
\item[$i, j, j'$] Index for buses
\item[$ij$]   Index for the branch between bus $i$ and $j$
\item[$m,n$]   Index for hydrogen nodes
\item[$mn$]   Index for the pipeline between nodes $m$ and $n$
\end{IEEEdescription}
\vspace{-5pt}
\subsection{Variables}
\begin{IEEEdescription}[\IEEEusemathlabelsep\IEEEsetlabelwidth{superscript}]
\addcontentsline{toc}{section}{Nomenclature}
%\item[$C^{\text{R/H/A}}$] Total costs of RG/HP/AS
%\item[$C_{\text{inv/ope}}^{\text{R/H/A}}$] Investment and operation costs of RG/HP/AS\vspace{1pt}
\item[$W^{\text{rg,wt/pv/bes/vc}}$] WT/PV/BES/VC installed capacities in the RG
\item[$W^{\text{hp,bes/ae/hst}}$]
%Quality Control Editor: Abbreviations and acronyms are often defined the first time the term is used within the abstract and again in the main text and then used throughout the remainder of the document. Please consider adhering to this convention. The target journal may have a list of abbreviations that are considered common enough that they do not need to be defined.
BES/AE/HST installation capacities in the HP
\item[$W^{\text{as,hst/asy/ast}}$] HST/ASY/AST installation capacities in the AS
\item[$\rho_{t}^{\text{rg-hp/as,e}}$] Electricity prices between RGs and the HP/ASs\vspace{1pt}
%\item[$\rho_{t}^{\text{rg-as,e}}$] Electricity price between RG and AS\vspace{1pt}
\item[$\rho_{t}^{\text{hp-as,h}}$] Hydrogen price between the HP and the AS
\item[$P_{t}^{\text{rg,wt}}, Q_t^{\text{rg,wt}}$]  Active and reactive power of WT \vspace{1pt}
\item[$P_{t}^{\text{rg,pv}}, Q_t^{\text{rg,pv}}$]  Active and reactive power of PV
\item[$P_{t}^{\text{rg,wt/pv,max}}$] Maximum power of WT/PV
\item[$P_{t}^{\text{rg,wt/pv,curt}}$] Power curtailment of WT/PV
\item[$P_{t}^{\text{rg/hp,bes,c/d}}$] BES charging/discharging power in the RG/HP
\item[$P_{t}^{\text{rg,sell,hp/as}}$] Power that the RG sell to the HP/AS \vspace{1pt}
\item[$P_{t}^{\text{hp/as,buy,rg}}$] Power bought by the HP/AS from the RG
\item[$P_{t}^{\text{hp,ae/comp}}$] Power of AE and hydrogen compressor \vspace{1pt}
\item[$P_{t}^{\text{as,asy}}$] Power consumption of ASY
\item[$P_{t}^{\text{as,back}}$] Backup power for continuous operation of ASY
\item[$P_{ij,t}, Q_{ij,t}$] Active/reactive power flows on branch $ij$
\item[$\ell_{ij,t}$]  Square of current on branch $ij$
\item[$\upsilon_{i,t}$]  Square of the voltage amplitude at \textcolor{black}{bus $i$}
\item[$p_{i,t}, q_{i,t}$] Active and reactive power injections at \textcolor{black}{bus $i$} \vspace{1pt}
\item[$Q_{t}^{\text{rg,bes/vc}}$] Reactive power of BES and VC in the RG \vspace{1pt}
\item[$Q_{t}^{\text{hp,bes}}$] Reactive power of BES in the HP
\item[$S^{\text{rg/hp,bes}}_t$]  State of BES in the RG/HP\vspace{1pt}
\item[$f^{\text{hp,pro}}_t$]  Hydrogen production rate\vspace{1pt}
\item[$f^{\text{hp,sell,as}}_t$]  Hydrogen sold from the HP to the AS
\item[$f^{\text{as,buy,hp}}_t$]  Hydrogen bought by the AS from the HP \vspace{1pt}
\item[$f^{\text{hp/as,hst,in/out}}_{t}$]  Hydrogen inflow/outflow of %the
HST in the HP/AS\vspace{1pt}
\item[$S^{\text{hp/as,hst}}_t$]  State of the HST in the HP/AS
\item[$F_{mn,t}$]  Average hydrogen flow of pipeline $mn$
\item[$F_{mn,t}^{\text{in/out}}$]  Hydrogen inflow/outflow of pipeline $mn$
\item[${\pi}_{m,t}$]  Pressure at hydrogen node $m$
%\item[$\tilde{\pi}_{mn,t}$]  Average pressure of hydrogen nodes $m$ and $n$
\item[$LP_{mn,t}$]  Linepack storage of pipeline $mn$
\item[$f^{\text{as,cons}}_t$]  Hydrogen consumption for ASY
\item[$M^{\text{as,pro}}_t$]  Flow rate of ammonia production
\item[$M^{\text{as,sell}}_t$]  Ammonia sold to the external market by the AS
\item[$S^{\text{as,ast}}_t$]  State of the AST
\end{IEEEdescription}
\vspace{-5pt}
\subsection{Parameters}
\begin{IEEEdescription}[\IEEEusemathlabelsep\IEEEsetlabelwidth{superscript}]
\addcontentsline{toc}{section}{Nomenclature}
\item[$T,\tau , \Delta t$] Planning/operational horizon and step length
\item[$ {\overline W}^{\text{rg,wt/pv/bes/vc}}$]  Upper capacity limit of WT/PV/BES/VC in the RG%\vspace{1pt}
\item[$ {\overline W}^{\text{hp,bes/ae/hst}}$]  Upper capacity limit of BES/AE/HST  in the HP%\vspace{1pt}
\item[$ {\overline W}^{\text{as,hst/asy/ast}}$] Upper capacity limit of HST/ASY/AST  in the AS%\vspace{1pt}
\item[$\underline { W}^{\text{rg,wt/pv/bes/vc}}$]  Lower capacity limit of WT/PV/BES/VC in the RG%\vspace{1pt}
\item[$\underline { W}^{\text{hp,bes/ae/hst}}$]  Lower capacity limit of BES/AE/HST in the HP%\vspace{1pt}
\item[$\underline { W}^{\text{as,hst/asy/ast}}$] Lower capacity limit of HST/ASY/AST in the AS%\vspace{1pt}
%\item[${\overline W}^{\text{rg,wt/pv/bes/vc}}$]   WT/PV/BES/VC capacity limits of RG\vspace{1pt}
%\item[${\overline W}^{\text{hp,bes/ae/hst}}$]  BES/AE/HST capacity limits of HP\vspace{1pt}
%\item[${\overline W}^{\text{as,hst/asy/ast}}$] HST/ASY/AST capacity limits of AS\vspace{1pt}
%\item[$\underline {W}^{\text{rg,wt/pv/bes/vc}}$]   WT/PV/BES/VC capacity limits of RG\vspace{1pt}
%\item[$\underline {W}^{\text{hp,bes/ae/hst}}$]  BES/AE/HST capacity limits of HP\vspace{1pt}
%\item[$\underline {W}^{\text{as,hst/asy/ast}}$] HST/ASY/AST capacity limits of AS\vspace{1pt}
\item[$c^{\text{wt/pv/bes/vc}}$] Unit investment costs of WT/PV/BES/VC
\item[$c^{\text{ae/hst/asy/ast}}$] Unit investment costs of AE/HST/ASY/AST
\item[$\eta^{\text{bes,c/d}}$] BES charging/discharging efficiencies
\item[$\overline{\eta}^{\text{bes}},\underline{\eta}^{\text{bes}}$] State limits of BES
\item[$\zeta^{\text{bes}},\sigma^{\text{deg}}$] BES self-discharge ratio and degradation cost
\item[$\overline{\upsilon}_{i},\underline{\upsilon}_{i}$] Voltage magnitude limits at bus $i$
\item[$\overline{\ell}_{ij}$] Current limit of branch $ij$
\item[$\overline{\eta}^{\text{ae}}, \underline{\eta}^{\text{ae}}$] Power limits of the hydrogen production plant
\item[$\eta^{\text{p2h}}$] Energy conversion coefficient of the AE
\item[$\eta^{\text{hp,comp}}$] Compressor power consumption coefficient
\item[$\eta^{\text{h2a}}$] ASY hydrogen consumption coefficient
\item[$\eta^{\text{p2a}}$] ASY power consumption coefficient
%\item[$\overline{f}^{\text{h-H,st,in/out}}$] Hydrogen inflow/outflow limits of hydrogen storage of HP
%\item[$\overline{f}^{\text{h-A,st,in/out}}$] Hydrogen inflow/outflow limits of hydrogen storage of AS
\item[$\overline{\eta}^{\text{hst}},\underline{\eta}^{\text{hst}}$] State limits of the HST
\item[$\underline{\pi}_{m},\overline{\pi}_{m}$] Limits of the squared hydrogen pressure at hydrogen node $m$
\item[$K_{mn}^{\text{lp}},K_{mn}^{\text{gf}}$] Weymouth constants of \textcolor{black}{pipeline} $mn$
\item[$\overline{\eta}^{\text{asy}},\underline{\eta}^{\text{asy}}$] Flow rate limits of ammonia production
\item[$\overline{r}^{\text{asy}},\underline{r}^{\text{asy}}$] Maximum ramping up and down limits of ASY
\item[$\overline{M}^{\text{as,sell}}$] Limit of ammonia sold to the external market
\item[$\rho_t^{\text{a}},\rho_t^{\text{as,back}}$] External ammonia price and backup power cost
\end{IEEEdescription}

\section{Introduction}
\label{sec:intro}

\subsection{Background and Motivation}
\label{sec:background}
\IEEEPARstart{I}{n} \textcolor{black}{recent years, wind and solar power installations have rapidly increased, accompanied by significant renewable energy sources (RESs) curtailment. Power to hydrogen (P2H) offers flexible adjustment capabilities, enabling extensive integration of RESs \cite{wang2024water}. According to the International Energy Agency (IEA), global installed electrolyzer capacity could reach 5 GW by 2024 and is expected to reach 230 GW by 2030 \cite{IEA2024}. Meanwhile, ammonia, an important global chemical material, is produced in hundreds of millions of tons annually \cite{macfarlane2020roadmap}. However, traditional fossil fuel-based ``grey ammonia" results in high carbon emissions \cite{macfarlane2020roadmap}.} Renewable power to ammonia (ReP2A) is a promising pathway for the large-scale utilization of RESs and \textcolor{black}{decarbonization} in the power and chemical industries \cite{yang2022breaking, macfarlane2020roadmap,guo2023deploying}. %Driven by the demand for decarbonization,
ReP2A projects have been implemented in many countries, including China \cite{neimenggu2022}, Denmark, and Australia \cite{campion2023techno}. \textcolor{black}{To the best of our knowledge, nearly 25 ReP2A projects were initiated in China in 2023 \cite{hydrogen2023}.}

\textcolor{black}{ReP2A systems can be classified into grid-connected and off-grid types based on their operational modes \cite{zeng2024scheduling,neimenggu2022}. Grid-connected systems are more stable but may face challenges such as failing to meet green certification \cite{giovanniello2024influence,yu2024optimal} due to the use of fossil fuel-based hydrogen, as well as high costs of connecting to the grid and operational limitations \cite{neimenggu2023}.}
Off-grid systems, in particular, have shown greater potential due to policy support \cite{2024} and flexibility in planning without grid integration constraints \cite{yu2024optimal}.

A key challenge in off-grid ReP2A systems is aligning ammonia synthesis (ASY), which has limited flexibility due to chemical safety requirements \cite{yu2023optimal}, with the variable hydrogen flow produced from renewable power. To address this, multistage buffer systems (BSs), including battery energy storage (BES), hydrogen storage tanks (HSTs), and ammonia storage tanks (ASTs), must be configured to ensure both economical and safe operations. The regulation durations of BES and HSTs typically range from hours to several days, whereas ASTs can span seasons. Investment in these BSs significantly impacts the operational performance of the ReP2A system. Proper planning of renewable power capacities, electrolyzers, and ASY is also crucial, as matching equipment capacities can increase their full load hours (FLHs) \cite{nayak2020techno}, thereby improving investment efficiency. In summary, coordinated sizing of the off-grid ReP2A system is essential for its technoeconomic performance and supports the development of the green hydrogen and ammonia industries \cite{nayak2020techno}.

However, involving multiple investors \cite{yu2023optimal} presents additional challenges in some projects.
For example, a renewable power-to-hydrogen (P2H) plant might be invested in by two stakeholders (renewable promoters and fertilizer producers) \cite{fernandez2022multilevel}.
In some cases, investments in ammonia plants do not include facilities for hydrogen production and renewable power generation \cite{coalchem2021}, leading to the ASY investor being a separate stakeholder. \textcolor{black}{Moreover, many emerging ReP2A projects have been jointly funded by multiple entities. In the project in Narisong, Inner Mongolia, China
\cite{sanxia2022}, the hydrogen plant is funded by China Three Gorges Corporation, while the ammonia plant is invested in by Inner Mongolia Yidong Group Jiuding Chemical Co., Ltd. The project in Ruijin, Jiangxi, China \cite{datang2024} is jointly invested by Datang Corporation Ltd., the local government, and HFG Hydrogen Energy Technology Co., Ltd. Other multistakeholder investment projects also include those in Paradip, India \cite{acwa_ammonia}, and Dubai, UAE \cite{jera_information}, which are not elaborated here.}
Consequently, the ReP2A system may involve three stakeholders: renewable power generation, hydrogen production, and ammonia synthesis (denoted RG, HP, and AS, respectively). These stakeholders often have conflicting interests, and their planning and operations influence each other.

\textcolor{black}{To guide the development of multistakeholder ReP2A systems, some regions in China have implemented policies \cite{neimenggu2022} that suggest an entity to invest in renewable generation and hydrogen production, while the ammonia plant can be invested in by another entity. However, such projects face challenges like the simultaneous initiation and commission of renewable generation and hydrogen production projects \cite{neimenggu2022}, and the need to determine downstream hydrogen application \cite{neimenggu2023}.
In regions without policy requirements, renewable generation and hydrogen production can be commissioned separately. Furthermore, recent policies \cite{neimenggu2024} in Inner Mongolia, China, allow the construction of multistakeholder ReP2A projects, and encourage stakeholders to sign long-term contracts. While the authority is guiding the development of ReP2A, there is still room for improvement of policies and measures.}

Existing works on ReP2A planning \cite{yu2023optimal,yu2024optimal} and operation \cite{wu2023multi} %, however,  %enhance the techno-economic feasibility of green ammonia. However, these studies
generally treat the system as a single entity \cite{yu2024optimal} or a collaborative model \cite{yu2023optimal}, overlooking the competing and conflicting interests among multiple stakeholders. To fill this gap, this paper focuses on the multistakeholder sizing problem %using game theory
in balancing the investments and profits of RGs, HPs, and ASs in ReP2A projects, reflecting real competition and dilemmas in engineering.

\subsection{Literature Review}
\label{sec:review}

In recent years, researchers have explored the planning and operation of ReP2A systems. For example, Wu et al. \cite{wu2023multi} proposed a method for grid-connected ReP2A systems to participate in electricity, hydrogen, and ammonia futures and spot markets as a virtual power plant (VPP). Yu et al. \cite{yu2023optimal} designed a two-stage optimal sizing and pricing method for grid-connected ReP2A systems, achieving globally optimal benefits while balancing the interests of multiple investors. Drawing on \cite{yu2023optimal}, Yu et al. \cite{yu2024optimal} proposed a mixed-integer linear fractional programming (MILFP) model to achieve a capacity configuration that minimizes the levelized cost of ammonia (LCOA) in off-grid ReP2A systems.

To address the spatial mismatch between renewable resources and ammonia demand, Li et al. \cite{li2022coplanning} proposed a collaborative planning model for ReP2A and the power grid, alleviating the burden of grid expansion through the hydrogen supply chain.
Zhao et al. \cite{zhao2021exploring} explored supply chain design and expansion for green ammonia over the next decade. %Case studies show that electricity prices and the cost of electrolyzers are key factors in cost reduction.
Zhao et al. \cite{zhao2022potential} further analyzed the benefits of green ammonia in reducing renewable power curtailment, highlighting that ReP2A investments offer economic and environmental advantages over cross-regional power transmission.
Dinh et al. \cite{dinh2024levelised} conducted a technoeconomic analysis on pipelines, liquefied fuel tankers, and high-voltage direct current (HVDC) transmissions, indicating their suitability for short, long, and medium distances, respectively.

 With respect to the diversified utilization of ammonia, Zhao et al. \cite{zhao2023co} demonstrated that green ammonia generation and storage can mitigate long-term fluctuations in renewables, reducing the marginal cost of energy storage (ES) and coal-fired carbon emissions.
Ik{\"a}heimo et al. \cite{ikaheimo2018power} explored the feasibility of ReP2A for supplying fertilizer raw materials and serving as ES and transport media in future 100\% renewable energy systems.
Klyapovskiy et al. \cite{klyapovskiy2021optimal} introduced an energy management system for electricity--hydrogen--ammonia industrial clusters, enhancing operational flexibility and efficiency through the demand response of ammonia plants.
Osman et al. \cite{osman2020scaling} addressed hydrogen storage challenges %by proposing industrial-scale simulation and optimization methods
for a 100\% renewable energy system using ammonia as a carrier.

However, the aforementioned studies consider the ReP2A system as planned and operated by a single entity. In the presence of multiple stakeholders and competition in transactions, unified planning and operation may not be feasible.
%Studies on multi-stakeholder equilibrium \cite{Ruiz2012Equilibria,naebi2020epec,Chen2021Conjectural} have addressed pricing and trading in day-ahead markets for traditional power systems \cite{Ruiz2012Equilibria}, multi-microgrids \cite{naebi2020epec}, and multi-energy systems \cite{Chen2021Conjectural}.
Research on multistakeholder equilibrium \cite{Ruiz2012Equilibria,naebi2020epec,Chen2021Conjectural} have addressed pricing and trading in day-ahead markets for  power systems \cite{Ruiz2012Equilibria}, microgrids \cite{naebi2020epec}, and multienergy systems \cite{Chen2021Conjectural,conejo2020operations}.
However, the ReP2A system differs significantly due to its heterogeneous flexibility among stakeholders with staged electricity--hydrogen--ammonia production processes and BSs. Additionally, existing research has focused primarily on operational aspects, with limited analysis of planning issues. Overall, there is currently insufficient research on multistakeholder sizing equilibrium (MSSE) in ReP2A systems %considering transaction competition,
leaving a gap in investment equilibrium in the green ammonia industry.

\subsection{Contributions}
\label{sec:contribution}

To fill the aforementioned gap, this paper analyzes the equilibrium for ReP2A multistakeholder sizing on the basis of noncooperative game theory. The main contributions of this study are summarized as follows:
\begin{enumerate}
	\item A
%Quality Control Editor: Abbreviations and acronyms typically need to be defined only once within the main text. Please consider adhering to this convention.
multistakeholder sizing equilibrium (MSSE)
model for ReP2A systems is proposed. This model encompasses the generation, storage, delivery, and consumption of electricity and hydrogen, with stakeholders simultaneously participating in transactions involving electricity, hydrogen, and ammonia.
	
	\item An equivalent optimization problem is constructed to solve the equilibrium. Moreover, the problem is decomposed in the temporal dimension and solved via a multicut generalized Benders decomposition (GBD) algorithm, thereby reducing the computational burden.
	%compared to direct approach when applied to large-scale planning.
	
	\item Through case studies on a realistic system, we find %that HP stakeholder holds a dominant position among the stakeholders,
	it difficult to simultaneously ensure the interests of all parties under free competition. We also find that a benefit transfer agreement may lead out of the dilemma.
\end{enumerate}

The rest of this paper is organized as follows. Section \ref{sec:structure} presents the structure of the ReP2A system. Sections \ref{sec:model} and \ref{sec:4method} elaborate on the proposed MSSE model and its solution methodology. Section \ref{sec:cases} presents the simulation results and discussion, and conclusions are drawn in Section \ref{sec:conclusions}.

\section{Structure of ReP2A Systems}
\label{sec:structure}

An off-grid ReP2A system, as presented in Fig. \ref{fig:system}, typically encompasses the generation, storage, delivery, and consumption of electricity and hydrogen alongside the production, storage, and sale of ammonia.
As discussed in Section \ref{sec:background}, these processes may belong to different stakeholders in real-life projects \cite{fernandez2022multilevel,coalchem2021}.
The stakeholders trade electricity, hydrogen, and ammonia with one another.%, achieving the circulation of energy and mass.

In the complete green ammonia production process, maintaining energy and mass balance is crucial. The RG, HP, and AS exhibit distinct flexibility due to their inherent physical characteristics.
% of renewable power, hydrogen production, and chemical ammonia synthesis.
To accommodate the heterogeneous flexibility, multi-type buffer systems (BSs) are implemented, with staged BES, HST, and AST designated for different timescales. \textcolor{black}{Additionally, the BES on the source side uses grid-forming technology \cite{zhu2024exploring} to provide frequency response, voltage support \cite{som2024an}, and enhance system strength \cite{liu2024system} ensuring the stable operation of the off-grid system.}

Without loss of generality, we make the following assumptions in analyzing the MSSE:

1) The RG, HP, and AS operate under individual rationality, aiming to minimize their own planning and operational costs.

2) Electricity and hydrogen prices are determined by the interaction between stakeholders without any additional constraints. The ammonia price follows the external market.

3) The RG and HP can both invest in their own BES. The HP and AS can both invest in HSTs.

4) The electricity transmission network \textcolor{black}{to HP and AS sides is invested in and owned by the RG, as policy recommendations \cite{neimenggu2024} suggest that renewable energy projects can be supported by accompanying transmission infrastructure, while} the hydrogen delivery network is owned by the HP. The network distances are determined by the geographical locations of the RG, HP, and AS.

5) The capacities of wind and solar power are predetermined based on local resources prior to planning.

\begin{figure}[t]
	\centering
	\includegraphics[width=3.45in]{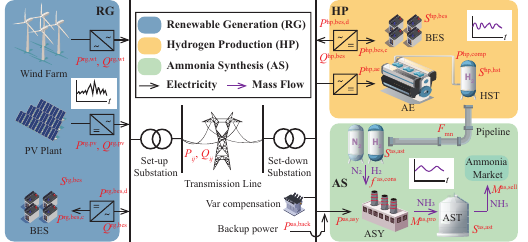}\vspace{-0pt}
	\caption{Schematic diagram of a typical off-grid ReP2A system.}
	\label{fig:system}\vspace{-0pt}
\end{figure}

\section{MSSE Model of the ReP2A System}
\label{sec:model}

This section provides a detailed modeling framework for the planning problems of the RG, HP, and AS. Furthermore, we outline the noncooperative game structure of the MSSE model for the ReP2A system.

\subsection{Modeling the Whole Process of ReP2A}
\label{sec:allprocess}

The entire ReP2A process involves the flow of energy and mass related to electricity, hydrogen, and ammonia.
\textcolor{black}{The BES is used for energy shifting, maintaining power balance, and supporting hydrogen production, while the HST ensures a stable hydrogen supply for ammonia synthesis, and the AST  smooths fluctuations in ammonia production. These systems gradually transform fluctuating renewable energy into stable ammonia production. Additionally, BES in the RG and HST in the HP are used to achieve more profits in transactions.}
The detailed physical models and planning and operational constraints of the corresponding stakeholders are as follows, where \textit{1)--3)} pertains to the RG, \textit{5)--8)} pertains to the HP, \textit{10)--12)} pertains to the AS, and \textit{4)} and \textit{9)} are the constraints coupling multiple stakeholders.

\subsubsection{RG Planning Capacities}
The equipment to be invested in by the RG includes BES and var compensation (VC) for managing active and reactive power. Their capacities are limited:
\begin{align}
	\hspace{-1pt}   \big[ \underline W^{\text{rg,bes}},\underline W^{\text{rg,vc}}  \big] \leq \big[ W^{\text{rg,bes}},W^{\text{rg,vc}}  \big] \leq \big[ \overline W^{\text{rg,bes}},\overline W^{\text{rg,vc}} \big]. \hspace{-2pt}  \label{eq:wb}
\end{align}

\subsubsection{BES Operation}
The BES provides both active and reactive power support. Its charging, discharging, and state of charge (SOC) are constrained by
\begin{align}
	& \big\Vert (P_{t}^{\text{rg,bes,c/d}} , Q_{t}^{\text{rg,bes}}) \big\Vert_2 \leq W^{\text{rg,bes}}, \label{eq:bes1}\\
	& \bm{0} \leq \big[ P_{t}^{\text{rg,bes,c}},P_{t}^{\text{rg,bes,d}} \big] \leq \big[ 0.5W^{\text{rg,bes}},0.5W^{\text{rg,bes}} \big],  \label{eq:bes2}  \\
	& S_{t}^{\text{rg,bes}}=(1-\zeta^{\text{bes}})S_{t-1}^{\text{rg,bes}} + (\eta^{\text{bes,c}}P_{t}^{\text{rg,bes,c}} - \dfrac {P_{t}^{\text{rg,bes,d}}}{\eta^{\text{bes,d}}})\Delta t, \label{eq:bes3}\\
	& S_{t=0}^{\text{rg,bes}}=S_{t=\tau}^{\text{rg,bes}}, \\
	& \underline{\eta}^{\text{bes}}W^{\text{rg,bes}} \leq S_{t}^{\text{rg,bes}}\leq\overline{\eta}^{\text{bes}}W^{\text{rg,bes}}. \label{eq:bes5}
\end{align}

 Moreover, its degradation cost is included in the objective with exact convex-relaxed constraints \cite{wang2024exact}; see Section \ref{sec:objrg}.

\subsubsection{Electrical Power Flow}
\textcolor{black}{The ReP2A network security constraints are met through power flow constraints. From the perspective of power flow direction and topology, the ReP2A system resembles a distribution network \cite{zeng2024scheduling}, as it features a radial network without a looped structure.}
The electrical network is described via \textcolor{black}{a branch flow model called DistFlow} \cite{steven2013branch}, as shown in  (\ref{eq:dist1})--(\ref{eq:dist5}). The power injection at each bus $i$ is determined by (\ref{eq:dist6})--(\ref{eq:dist9}). \textcolor{black}{We do not use the DC power flow, which is commonly applied in the transmission network, because the DistFlow model is more accurate and it has manageable computational complexity.}
\begin{align}
	&\textstyle \sum_{j':j\to j'}P_{jj',t}=p_{j,t}+\sum_{i:i\to j}P_{ij,t}-r_{ij}\ell_{ij,t}, \label{eq:dist1}\\
	&\textstyle \sum_{j':j\to j'}Q_{jj',t}=q_{j,t}+\sum_{i:i\to j}Q_{ij,t}-x_{ij}\ell_{ij,t}, \\
	&\upsilon_{j,t}=\upsilon_{i,t}-2r_{ij}P_{ij,t}-2x_{ij}Q_{ij,t}+(r_{ij}^{2}+x_{ij}^{2})\ell_{ij,t}, \\
	&\left\|\left(2P_{ij,t},2Q_{ij,t},\ell_{ij,t}-\upsilon_{i,t}\right)\right\|_{2}\leq\ell_{ij,t}+\upsilon_{i,t}, \\
	&\underline{\upsilon}_{i} \leq \upsilon_{i,t}\leq\overline{\upsilon}_{i},\ \ell_{ij,t}\leq\overline{\ell}_{ij},  \label{eq:dist5}\\
	&p_{i,t}=P_{i,t}^{\text{rg,wt}}+P_{i,t}^{\text{rg,pv}}+P_{i,t}^{\text{rg,bes,d}} \nonumber \\
	& \hspace{30pt} -P_{i,t}^{\text{rg,bes,c}}-P_{i,t}^{\text{rg,sell,hp}}-P_{i,t}^{\text{rg,sell,as}}, \label{eq:dist6}\\
	&q_{i,t}=Q_{i,t}^{\text{rg,wt}}+Q_{i,t}^{\text{rg,pv}}+Q_{i,t}^{\text{rg,vc}}+Q_{i,t}^{\text{rg,bes}}, \label{eq:dist7}\\
	& P_{t}^{\text{rg,wt/pv}}=P_{t}^{\text{rg,wt/pv,max}}-P_{t}^{\text{rg,wt/pv,curt}}, \\
	& \big\Vert (P_{t}^{\text{rg,wt/pv}} , Q_{t}^{\text{rg,wt/pv}})\big\Vert_2 \leq W^{\text{rg,wt/pv}},\\
	&|Q_{t}^{\text{rg,vc}}|\leq W^{\text{rg,vc}}.\label{eq:dist9}
\end{align}

\subsubsection{Electricity Transactions}
The RG stakeholder sells renewable electricity to the HP and AS stakeholders. When a transaction is completed, the quantity and price satisfy:
\begin{align}\hspace{-0pt}
	\begin{cases}
		P_{t}^{\text{rg,sell,hp}}-P_t^{\text{hp,buy,rg}}=0:(\rho_t^{\text{rg-hp,e,rg}},\rho_t^{\text{rg-hp,e,hp}},\rho_t^{\text{rg-hp,e,as}}),\\
		P_{t}^{\text{rg,sell,as}}-P_t^{\text{as,buy,rg}}=0:(\rho_t^{\text{rg-as,e,rg}},\rho_t^{\text{rg-as,e,hp}},\rho_t^{\text{rg-as,e,as}}),  \label{eq:market1}
	\end{cases}\hspace{-0pt}
\end{align}
%which serves as the one of the coupling constraints in the planning problem, with
where $\rho_t^{\text{rg-hp,e,rg/hp/as}}$ and $\rho_t^{\text{rg-as,e,rg/hp/as}}$ are the dual variables for the individual planning problems of %the
RG/HP/AS, respectively.

\subsubsection{HP Planning Capacities}
The HP can invest in alkaline electrolyzers (AEs), HSTs, and its own BESs, as shown in Fig. \ref{fig:system}. The planning capacities of the AE and HST are limited by
\begin{align}
	& \big[ \underline W^{\text{hp,ae}},\underline W^{\text{hp,hst}} \big] \leq \big[W^{\text{hp,ae}},W^{\text{hp,hst}}\big] \leq \big[\overline W^{\text{hp,ae}},\overline W^{\text{hp,hst}} \big],  \label{eq:wel}
	%&W_{\text{min}}^{\text{h-H,st}}\leq W^{\text{h-H,st}}\leq W_{\text{max}}^{\text{h-H,st}}, \label{eq:whs}\\
\end{align}
and the planning and operational constraints of the BES share the same forms of (\ref{eq:wb}) and  (\ref{eq:bes1})--(\ref{eq:bes5}), with decision variables replaced by $W^{\text{hp,bes}}$, $P_{t}^{\text{hp,bes,c/d}}$, $Q_{t}^{\text{hp,bes}}$, and $S_{t}^{\text{hp,bes}}$.

\subsubsection{P2H Operation}
The linear model (\ref{eq:p2h}) from \cite{li2022coplanning} is used to approximate the hydrogen production of the industry-scale hydrogen plant.
%\cite{qiu2023extend}.
The load range follows (\ref{eq:p2hmaxmin}). Before storage and delivery, %the produced
hydrogen is pressurized, and the power consumption follows (\ref{eq:hcomp}). Additionally, the HP stakeholder may also invest in BES to assist in %the operation of the AEs, safe production
maintaining power balance (\ref{eq:hbanlance}) at the load bus and managing electricity transactions with the RG.
\begin{align}
	&f_{t}^{\text{hp,pro}}=P_{t}^{\text{hp,ae}}\eta^{\text{p2h}}, \label{eq:p2h}\\
	&\underline{\eta}^{\text{ae}}W^{\text{hp,ae}}\leq P_{t}^{\text{hp,ae}}\leq\overline{\eta}^{\text{ae}}W^{\text{hp,ae}},  \label{eq:p2hmaxmin}\\
	&P_{t}^{\text{hp,comp}}=f_{t}^{\text{hp,pro}}\eta^{\text{comp}}, \label{eq:hcomp}\\
	&P_{t}^{\text{hp,buy,rg}}+P_{t}^{\text{hp,bes,d}}=P_{t}^{\text{hp,bes,c}}+P_{t}^{\text{hp,ae}}+P_{t}^{\text{hp,comp}}. \label{eq:hbanlance}
\end{align}

\subsubsection{HST Operation}
%The HP stakeholder invests in the HST, aiming to
The HST invested in by the HP stakeholder is used to
accommodate the varying hydrogen flow and maximize the benefit from the hydrogen transactions with the AS stakeholder. The state of the HST satisfies (\ref{eq:hes1})--(\ref{eq:hes3}), with hydrogen charging and release \textcolor{black}{being} constrained by (\ref{eq:hes4}).
\begin{align}
	&S_{t+1}^{\text{hp,hst}} =S_{t}^{\text{hp,hst}}+(f_{t}^{\text{hp,hst,in}}-f_{t}^{\text{hp,hst,out}})\Delta t, \label{eq:hes1}\\
	&S_{t=0}^{\text{hp,hst}}=S_{t=T}^{\text{hp,hst}}, \\
	&\underline{\eta}^{\text{h}}W^{\text{hp,hst}}\leq S_{t}^{\text{hp,hst}}\leq\overline{\eta}^{\text{h}}W^{\text{hp,hst}}, \label{eq:hes3} \\
	&\bm{0}\leq [f_{t}^{\text{hp,hst,in}},f_{t}^{\text{hp,hst,out}}]\leq[0.5W^{\text{hp,hst}}, 0.5W^{\text{hp,hst}}].    \label{eq:hes4}
\end{align}

\subsubsection{Hydrogen Delivery Network}
The hydrogen flow in pipelines can be modelled by the modified Weymouth equation with second-order cone relaxation \cite{an2024decomposition}, as %shown in
(\ref{eq:weymouth})--(\ref{eq:pressure}). Linepack  dynamics are shown in (\ref{eq:lp})--(\ref{eq:linepack}), and linepack recycling conditions are presented in (\ref{eq:lp-balance}) for each operational horizon. The balance of hydrogen flow in the network is constrained by (\ref{eq:hydrogenbalance}).
\begin{align}
	&(F_{mn,t} / K_{mn}^{\text{gf}})^{2}\leq\pi_{m,t}^{2}-\pi_{n,t}^{2}, \label{eq:weymouth}\\
	&F_{mn,t}=(F_{mn,t}^{\text{in}}+F_{mn,t}^{\text{out}})/2, ~F_{mn,t}\geq0 ,  \label{eq:hydrogenflow}\\
	&\underline{\pi}_{m}\leq\pi_{m,t}\leq\bar{\pi}_{m},  \label{eq:pressure}\\
	&LP_{mn,t}=K_{mn}^{\text{lp}}(\pi_{m,t}+\pi_{n,t})/2, \label{eq:lp}\\
	&LP_{mn,t+1}=LP_{mn,t}+F_{mn,t}^{\text{in}}-F_{mn,t}^{\text{out}}, \label{eq:linepack}\\
	%&\tilde{\pi}_{mn,t}=(\pi_{m,t}+\pi_{n,t})/2, \label{eq:averagepressure}\\
	&LP_{mn,t=0}=LP_{mn,t=\tau}, \label{eq:lp-balance}\\
	 &f_{t}^{\text{hp,pro}}+f_{t}^{\text{hp,hst,out}}-f_{t}^{\text{hp,hst,in}}+F_{mn,t}^{\text{out}}-F_{mn,t}^{\text{in}}=f_{t}^{\text{hp,sell,as}}.\hspace{-5pt} \label{eq:hydrogenbalance}
\end{align}

\subsubsection{Hydrogen Transaction}
The HP stakeholder sells hydrogen to the AS, and the clearing condition of the transaction is
\begin{align}
	f_t^{\text{hp,sell,as}}-f_t^{\text{as,buy,hp}}=0:(\rho_t^{\text{hp-as,h,rg}},\rho_t^{\text{hp-as,h,hp}},\rho_t^{\text{hp-as,h,as}}), \label{eq:market2}
\end{align}
where $\rho_t^{\text{hp-as,h,rg/hp/as}}$ are dual variables. % for the individual planning problems.

\subsubsection{AS Planning Capacities}
The AS stakeholder may invest in ASY, AST, and its own HST to accommodate the variable hydrogen supply and help bargain with the HP stakeholder. The capacities of these facilities are limited by
%(\ref{eq:Was})-(\ref{eq:Whst2}). %AS configures HST to smooth out hydrogen flow fluctuations, ensuring stable production for ASY.
\begin{align}
	&[\underline W^{\text{as,asy}},\underline W^{\text{as,ast}}]\leq [W^{\text{as,asy}},W^{\text{as,ast}}]\leq [\overline W^{\text{as,asy}},\overline W^{\text{as,ast}}]. \label{eq:Was}
	%&W_{\text{min}}^{\text{a,st}}\leq W^{\text{a,st}}\leq W_{\text{max}}^{\text{a,st}}, \label{eq:Wast}\\
	%  &(\ref{eq:wel}), (\ref{eq:hes1})-(\ref{eq:hes4})~\text{for}~\{W^{\text{as,hst}},f_{t}^{\text{as,hst,in/out}},S_{t}^{\text{as,hst}}\}. \label{eq:Whst2}
\end{align}

The HST operational constraints of the AS share the same forms of (\ref{eq:wel}) and (\ref{eq:hes1})--(\ref{eq:hes4}), with the decision variables replaced by $W^{\text{as,hst}}$, $f_{t}^{\text{as,hst,in/out}}$, and $S_{t}^{\text{as,hst}}$.
% summarized in (\ref{eq:Whst2}).

\subsubsection{ASY Operation}
The ASY produces ammonia from hydrogen bought from the HP and nitrogen separated from the air via the Haber--Bosch process \cite{campion2023techno}. %This process consumes electricity for air separation and other auxiliary equipment.
Its hydrogen and power consumption (for air separation and other auxiliary equipment) satisfy (\ref{eq:h2a}) and (\ref{eq:p2a}) \cite{li2022coplanning}.
The electrical power comes not only from the RG but also from the backup power supply, as shown in (\ref{eq:pas}), to ensure chemical process safety. The operating range and ramping limits of ASY, constrained by chemical safety considerations \cite{klyapovskiy2021optimal}, are shown in (\ref{eq:aslimit})--(\ref{eq:asramp}), and the balance of hydrogen is depicted in (\ref{eq:hydrogenbalance2}).
\begin{align}
	&M_{t}^{\text{as,pro}}=f_{t}^{\text{as,cons}}\eta^{\text{h2a}},  \label{eq:h2a}\\
	&M_{t}^{\text{as,pro}}=P_{t}^{\text{as,asy}}\eta^{\text{p2a}},  \label{eq:p2a}\\
	&P_{t}^{\text{as,back}}+P_{t}^{\text{as,buy,rg}}=P_{t}^{\text{as,asy}},  \label{eq:pas}\\
	&\underline{\eta}^{\text{asy}}W^{\text{as,asy}}\leq M_{t}^{\text{as,pro}}\leq\overline{\eta}^{\text{asy}}W^{\text{as,asy}}, \label{eq:aslimit}\\
	-&\underline{r}^{\text{asy}}W^{\text{as,asy}}\leq M_{t+1}^{\text{as,pro}}-M_{t}^{\text{as,pro}}\leq\overline{r}^{\text{asy}}W^{\text{as,asy}}, \label{eq:asramp}\\
	&f_{t}^{\text{as,buy,hp}}+f_{t}^{\text{as,hst,out}}=f_{t}^{\text{as,hst,in}}+f_{t}^{\text{as,cons}}. \label{eq:hydrogenbalance2}
\end{align}

\subsubsection{AST Operation}
%The stakeholder of AS invests in the AST
The AST invested in by the AS is used
to accommodate the varying yield of green ammonia and manage the sales to the external ammonia market under fluctuating prices. The operational constraints of the AST include
\begin{align}
	&S_{t+1}^{{\text{as,ast}}}=S_{t}^{{\text{as,ast}}}+(M_{t}^{{\text{as,pro}}}-M_{t}^{{\text{as,sell}}})\Delta t,\\
	&S_{t=0}^{{\text{as,ast}}}=S_{t=T}^{{\text{as,ast}}},\\
	&0\leq S_{t}^{{\text{as,ast}}}\leq W^{{\text{as,ast}}},\\
	&0\leq M_{t}^{{\text{as,sell}}}\leq\overline{M}^{\text{as,sell}}. \label{eq:amsell}
\end{align}

\begin{table*}[tb]\scriptsize
	\renewcommand{\arraystretch}{1.6}
	\caption{Summary of the Multistakeholder Planning Problem of the ReP2A System}\vspace{-0pt}
	\label{tab:summary}
	\setlength{\tabcolsep}{3pt}
	\centering
	\begin{tabular}{ccccc}
		\hline\hline
		\vspace{-10pt}
		\\
		Stakeholder   & Objective 	&\tabincell{c}{Independent\\constraints}    & Decision variables    &\tabincell{c}{Coupling constraints}  \\
		\hline
		RG            & (\ref{eq:CR})   & (\ref{eq:wb})--(\ref{eq:dist9})       &  $\bm{x}^{\text{rg}} \triangleq \{W^{\text{rg,bes/vc}},P^{\text{rg,wt/pv}},P^{\text{rg,wt/pv,curt}},P^{\text{rg,sell,hp,as}},P^{\text{rg,bes,c/d}},S^{\text{rg,bes}},Q^{\text{rg,wt/pv/bes/vc}},P_{ij}, Q_{ij}, \ell_{ij}, \upsilon_{i},p_{i}, q_{i}$\}    & \multirow{3}{*}{(\ref{eq:market1}), (\ref{eq:market2}) }    \\
		HP            & (\ref{eq:CH})   & (\ref{eq:wel})--(\ref{eq:hydrogenbalance})      & $\bm{x}^{\text{hp}} \triangleq \{W^{\text{hp,bes/ae/hst}},P^{\text{hp,ae/comp}},P^{\text{hp,buy,rg}},P^{\text{hp,bes,c/d}},S^{\text{hp,hst/ast}},f^{\text{hp,pro}},f^{\text{hp,sell,as}},f^{\text{hp,hst,in/out}},F_{mn}^{\text{in/out}},LP_{mn},{\pi}_{m}\}$    &   \\
		AS            & (\ref{eq:CA})   &  (\ref{eq:Was})--(\ref{eq:amsell})  & $\bm{x}^{\text{as}} \triangleq \{W^{\text{as,hst/asy/ast}},P^{\text{as,asy/back}},P^{\text{as,buy,rg}},f^{\text{as,buy,hp}},f^{\text{as,hst,in/out}},M^{\text{as,pro/sell}},S^{\text{as,ast}}\}$
		&     \\
		\hline\hline
	\end{tabular}\vspace{-0pt}
\end{table*}

\subsection{Multi-Stakeholder Sizing Model}
\label{sec:obj}

In the planning and operation phases, the RG, HP, and AS stakeholders interact with one another, requiring that the competing economic benefits of each stakeholder be considered, which differs from planning the ReP2A system as a single entity as reported in the literature. To characterize these interactions, a multistakeholder sizing model is proposed. The mathematical expressions are given as follows.

\subsubsection{RG}
\label{sec:objrg}

The objective of RG planning is to minimize its total cost $C^\text{rg}$,
\begin{align}
	\min\ C^\text{rg}=C_{\text{inv}}^\text{rg}+C_{\text{ope}}^\text{rg}, \label{eq:CR}
\end{align}

\noindent
which includes the investment cost %$C_{\text{inv}}^{\text{rg}}$
\begin{align}
	\nonumber C_{\text{inv}}^{\text{rg}}=&\text{CRF}(r,y)(c^\text{wt}W^\text{rg,wt}+c^\text{pv}W^\text{rg,pv}+c^\text{bes}W^\text{rg,bes}\\
	&+c^\text{vc}W^\text{rg,vc}+C^{\text{rg,line}}),
\end{align}
and the operational cost %$C_{\text{ope}}^{\text{rg}}$
\begin{align}
	\nonumber C_{\text{ope}}^{\text{rg}}=&\frac{8760}T\sum\nolimits_{t=1}^T(\sigma^{\text{deg}}P_{t}^{\text{rg,bes,d}}-P_{t}^{\text{rg,sell,hp}}\rho_t^{\text{rg-hp,e,rg}})\Delta t\\
	&-P_{t}^{\text{rg,sell,as}}\rho_t^{\text{rg-as,e,rg}}\Delta t
	+{\eta}_{\text{O\&M}}C_{\text{inv}}^{\text{rg}},
\end{align}
where $\text{CRF}(r,y)={r(1+r)^y}/{(\left(1+r\right)^y-1)}$ is the capital recovery factor; $r=8\%$ and $y$ are the discount rate and lifetime; %, respectively;
$C^{\text{rg,line}}$ is the transmission line cost; and ${\eta}_{\text{O\&M}}=2\%$ is the annual operation and maintenance (O\&M) cost factor. \textcolor{black}{Here, we do not plan the power lines and hydrogen pipelines; instead, their fixed costs are included in the cost function according to assumption 4).}

\subsubsection{HP}
\label{sec:objhp}

The planning goal of the HP stakeholder is to minimize its total cost
\begin{align}
	\min\ C^\text{hp}=C_{\text{inv}}^\text{hp}+C_{\text{ope}}^\text{hp}, \label{eq:CH}
\end{align}
which includes investment cost $C_{\text{inv}}^{\text{hp}}$ and operational cost $C_{\text{ope}}^{\text{hp}}$:
\begin{align}
	\nonumber C_{\text{inv}}^{\text{hp}}=&\text{CRF}(r,y) \big(c^{\text{ae}}W^{\text{hp,ae}}+c^{\text{hst}}W^{\text{hp,hst}}+c^\text{bes}W^\text{hp,bes}\\
	&+C^{\text{hp,pipe}} \big),\\
	\nonumber   C_{\text{ope}}^{\text{hp}}=& \frac{8760}{T}  \sum\nolimits_{t=1}^T \big(P_t^{\text{hp,buy,rg}}\rho_t^{\text{rg-hp,e,hp}}+\sigma^{\text{deg}}P_{t}^{\text{hp,bes,d}} \big)\Delta t \\ &-f_t^{\text{hp,sell,as}}\rho_t^{\text{hp-as,h,hp}}\Delta t+{\eta}_{\text{O\&M}}C_{\text{inv}}^{\text{hp}},
\end{align}
where $C^{\text{hp,pipe}}$ is the hydrogen pipeline cost.

\subsubsection{AS}
The planning goal of the AS stakeholder is to minimize its total cost:
\begin{align}
	\min\ C^\text{as}=C_{\text{inv}}^\text{as}+C_{\text{ope}}^\text{as}, \label{eq:CA}
\end{align}
where the investment cost $C_{\text{inv}}^{\text{as}}$ and operation cost $C_{\text{ope}}^{\text{as}}$ follow
\begin{align}
	\hspace{-3pt}
	C_{\text{inv}}^\text{as}=&\text{CRF}(r,y)(c^\text{as}W^\text{as,asy}+c^\text{ast}W^\text{as,ast}+c^{\text{hst}}W^{\text{as,hst}}),\\
	\hspace{-3pt}
	\nonumber C_{\text{ope}}^{\text{as}}=&\frac{8760}T \sum\nolimits_{t=1}^T \big[(f_t^{\text{as,buy,hp}}\rho_t^{\text{hp-as,h,as}}+P_t^{\text{as,buy,rg}}\rho_t^{\text{rg-as,e,as}})\Delta t\\
	&
	+P_t^{\text{as,back}}\rho_t^{\text{as,back}}\Delta t-M_t^{\text{as,sell}}\rho_t^{\text{a}}\Delta t \big] +{\eta}_{\text{O\&M}}C_{\text{inv}}^{\text{as}}. \label{eq:CAope}
\end{align}

Given the above, we summarize the multistakeholder planning problem in Table \ref{tab:summary}.
For simplicity, \textcolor{black}{we use $(\bm{x}^k)_{k\in\{\text{rg,hp,as}\}}$, $(\bm{h}^k)_{k\in\{\text{rg,hp,as}\}}$, and $(\bm{g}^k)_{k\in\{\text{rg,hp,as}\}}$ to denote all independent decision variables, equality constraints, and inequality constraints of the RG, HP, and AS, respectively. Following this, the strategy set for each stakeholder can be defined as
\begin{align}
X^k=\{\bm{x}^k:\bm{h}^k(\bm{x}^k)=\bm{0},\bm{g}^k(\bm{x}^k)\leq \bm{0}\},~\forall k \in \{\text{rg,hp,as}\}.
\end{align}}

Considering the individual rationality of each stakeholder and competition in the transactions, we frame the multi-stakeholder sizing problem as a non-cooperative game and propose an MSSE model for the ReP2A system, with the game structure depicted in Fig. \ref{fig:game}. \textcolor{black}{The game is defined as follows:
\begin{itemize}
	 \item Player: RG, HP, and AS, denoted by \{rg,hp,as\}.
     \item Strategy: all decision variables $\bm{x}^k$, where $\bm{x}^k \in X^k, k\in \{\text{rg,hp,as}\}$.
     \item Payoff: objective function $C^k$, $k\in \{\text{rg,hp,as}\}$.
\end{itemize}}

\textcolor{black}{The competition between the RG, HP, and AS is relatively equal, which is different from the relationship between the power aggregator and electricity users \cite{tavakkoli2021bonus}, or between hydrogen refueling stations and hydrogen fuel cell vehicles \cite{zhang2023coordinated}. In analogy, it is more similar to the relationship between power, nature-gas, \cite{Chen2021Conjectural,conejo2020operations} and heating systems \cite{chen2022strategic}. Therefore, there is no absolute leader and follower, and it is modeled as a Nash game \cite{Chen2021Conjectural,conejo2020operations,chen2022strategic} rather than a Stackelberg game \cite{tavakkoli2021bonus, zhang2023coordinated}.}

According to (\ref{eq:CR})-(\ref{eq:CAope}), there are conflicts of interest among RG, HP, and AS, particularly in terms of the revenues and costs related to electricity and hydrogen trading, which will be discussed in Section \ref{sec:planning results}. Moreover, the investment costs of each stakeholder are not balanced due to heterogeneous physical characteristics in the staged processes, and we discuss this later in Sections \ref{sec:planning} and \ref{sec:extension}.

\begin{figure}[t]
	\centering
	\includegraphics[width=3.45in]{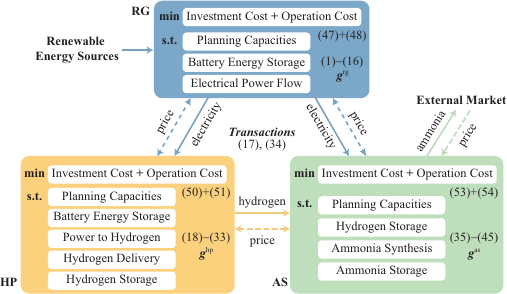}\vspace{-0pt}
	\caption{Game structure of MSSE in planning the ReP2A system.}
	\label{fig:game}\vspace{-0pt}
\end{figure}

\section{Solution Method for the MSSE Model}
\label{sec:4method}

To obtain the equilibrium of the proposed MSSE model, this section transforms the system of equations for solving equilibrium, using the direct approach, into a convex optimization problem. A temporal decomposition-based algorithm is subsequently developed to solve this problem efficiently.

\subsection{Direct Approach Based on KKT Conditions}
\label{sec:4.1kkt}

The equilibrium of noncooperative games is usually solved via fixed-point iteration \cite{li2022noncooperative} and direct approaches \cite{pan2021investment,Chen2021Conjectural}, including those based on Karush--Kuhn--Tucker (KKT) conditions \cite{pan2021investment} and strong duality \cite{Chen2021Conjectural}. \textcolor{black}{However, in the iteration method, the numerous interaction variables between stakeholders in planning problems make it difficult to ensure convergence. Additionally, the solution time for each individual stakeholder cannot be guaranteed and may even be infeasible. Thus, due to the slow convergence of fixed-point iteration, it may not be suitable for solving MSSE in this paper. Moreover,} strong duality introduces bilinear terms, making it challenging to solve. Thus, we attempt to solve the MSSE model via KKT conditions.

The RG, HP, and AS models established in Section \ref{sec:model} are all convex. The KKT conditions $\Omega^{k}$ for each stakeholder $\forall k \in \{ \text{rg,hp,as}\}$ are given as follows:\textcolor{black}{
\begin{subequations}
	\begin{align}
		\nonumber &\nabla_{\bm{x}^{k}}C^{k}+(\bm{\lambda}^{k})^{\text{T}}\nabla_{\bm{x}^{k}}\bm{g}^{k}+(\bm{\mu}^{k})^{\text{T}}\nabla_{\bm{x}^{k}}\bm{h}^{k}\\
        & +(\bm{\rho}^{\text{rg-hp,e},k})^{\text{T}}\nabla_{\bm{x}^{k}}(P^{\text{rg,sell,hp}}-P^{\text{hp,buy,rg}}) \nonumber\\
		\nonumber&+(\bm{\rho}^{\text{rg-as,e},k})^{\text{T}}\nabla_{\bm{x}^{k}}(P^{\text{rg,sell,as}}-P^{\text{as,buy,rg}})\\
		&+(\bm{\rho}^{\text{hp-as,h},k})^{\text{T}}\nabla_{\bm{x}^{k}}(f^{\text{hp,sell,as}}-f^{\text{as,buy,hp}}) = \bm{0} ,\\
		&(\bm{\lambda}^{k})^{\text{T}}\bm{g}^{k} = \bm{0}, ~\bm{\lambda}^{k} \geq \bm{0},~ \bm{g}^{k}\leq \bm{0},  \label{eq:complementary1} \\
        &\bm{h}^{k}=\bm{0},  \\
		&\text{(\ref{eq:market1})},\ \text{(\ref{eq:market2})},
  \end{align}\label{eq:equilibrium}\end{subequations}
where $\bm{\lambda}^{k}$ and $\bm{\mu}^{k}$ are the dual variables of $\bm{g}^{k}$ and $\bm{h}^{k}$. There are 9 dual variables in total, with 6 in (\ref{eq:market1}) and 3 in (\ref{eq:market2}). Among them, there are 3 for RG, 3 for HP, and 3 for AS. In summary, each KKT condition contains 3 dual variables.}

The solution of the joint KKT conditions $\{ \Omega^{\text{rg}}, \Omega^{\text{hp}}, \Omega^{\text{as}} \}$ is the equilibrium \cite{dreves2011solution}. However, dealing with the complementary constraints (\ref{eq:complementary1}) in $\Omega^{k}$ introduces many 0--1 variables when applied to complex problems, making them difficult to solve. Therefore, we construct an equivalent optimization problem \cite{egging2020solving,wang2021distribute}, formed as second-order cone programming (SOCP), to address the joint KKT conditions.

\subsection{Equivalent Problem for Solving the Equilibrium}
\label{sec:4.2proof}

The dual variables \textcolor{black}{of} (\ref{eq:market1}) and (\ref{eq:market2}) represent the equilibrium prices for different stakeholders. When transactions are reached, trading prices w.r.t. all parties $\forall k\in\{\text{rg,hp,as}\}$ are the same, as
\begin{align}
  \bm{\rho}^{\text{rg-hp,e}}=\bm{\rho}^{\text{rg-hp,e,}k},\bm{\rho}^{\text{rg-as,e}}=\bm{\rho}^{\text{rg-as,e,}k},\bm{\rho}^{\text{hp-as,h}}=\bm{\rho}^{\text{hp-as,h,}k}. \label{eq:ve}
\end{align}
The equilibrium under constraint (\ref{eq:ve}) is known as \emph{variational equilibrium (VE)} \cite{Chen2021Conjectural}.

\begin{theorem}
	\label{theorem}
	The VE of (\ref{eq:equilibrium}) is equivalent to the optimal solution of problem (\ref{eq:newproblem}), which is
\textcolor{black}{
	\begin{align}
		\nonumber \min_{\bm{x}^{k},\forall k} ~&  \sum\nolimits_{k\in \{\text{rg,hp,as}\}} C^{k}, \\
		\nonumber  \text{ s.t.}  ~~ & \bm{h}^{k} = \bm{0} {\colon\bm{\mu}^{k}},~\forall k \in \{\text{rg,hp,as}\}, \\
        \nonumber  & \bm{g}^{k}\leq \bm{0} {\colon\bm{\lambda}^{k}},~\forall k \in \{\text{rg,hp,as}\}, \\
		&\text{(\ref{eq:market1})}: \big[\bm{\rho}^{\text{rg-hp,e}},\bm{\rho}^{\text{rg-as,e}}\big], ~\text{(\ref{eq:market2})}:\bm{\rho}^{\text{hp-as,h}}. \label{eq:newproblem}
	\end{align}}\noindent
\end{theorem}
\begin{proof}
	The nonconvex terms in the original expression of the objective of (\ref{eq:newproblem}) include $P_{t}^{\text{hp,buy,rg}}\rho_{t}^{\text{rg-hp,e,hp}}-P_{t}^{\text{rg,sell,hp}}\rho_{t}^{\text{rg-hp,e,rg}}$, $P_{t}^{\text{as,buy,rg}}\rho_{t}^{\text{rg-as,e,as}}-P_{t}^{\text{rg,sell,as}}\rho_{t}^{\text{rg-as,e,rg}}$, and $f_{t}^{\text{hp,buy,as}}\rho_{t}^{\text{hp-as,h,hp}}-f_{t}^{\text{as,sell,hp}}\rho_{t}^{\text{hp-as,h,as}}$. Substituting (\ref{eq:market1}), (\ref{eq:market2}), and (\ref{eq:ve}) into these nonconvex terms yields zero, indicating that the objective of (\ref{eq:newproblem}) is substantially linear. Hence, problem (\ref{eq:newproblem}) is convex. Its optimality follows the KKT conditions:\textcolor{black}{
	\begin{subequations}
		\begin{align}
	\nonumber & \partial{L}/\partial \bm{x}^\text{rg} = \nabla_{\bm{x}^{\text{rg}}}(\textstyle \sum_{k} C^{k})+(\bm{\lambda}^{\text{rg}})^{\text{T}}\nabla_{\bm{x}^{\text{rg}}} \bm{g}^{\text{rg}}+(\bm{\mu}^{\text{rg}})^{\text{T}}\nabla_{\bm{x}^{\text{rg}}} \bm{h}^{\text{rg}}\\
			\nonumber&+(\bm{\lambda}^{\text{hp}})^{\text{T}}\nabla_{\bm{x}^{\text{rg}}} \bm{g}^{\text{hp}}+(\bm{\mu}^{\text{hp}})^{\text{T}}\nabla_{\bm{x}^{\text{rg}}} \bm{h}^{\text{hp}}\\
			\nonumber&+(\bm{\lambda}^{\text{as}})^{\text{T}}\nabla_{\bm{x}^{\text{rg}}}\bm{g}^{\text{as}}+
            (\bm{\mu}^{\text{as}})^{\text{T}}\nabla_{\bm{x}^{\text{rg}}}\bm{h}^{\text{as}}\\
			\nonumber&+(\bm{\rho}^{\text{rg-hp,e}})^{\text{T}}\nabla_{\bm{x}^{\text{rg}}}(P^{\text{rg,sell,hp}}-P^{\text{hp,buy,rg}}) \label{eq:L/xR} \\
			\nonumber &+(\bm{\rho}^{\text{rg-as,e}})^{\text{T}}\nabla_{\bm{x}^{\text{rg}}}(P^{\text{rg,sell,as}}-P^{\text{as,buy,rg}})\\
			&+{(\bm{\rho}^{\text{hp-as,h}})}^{\text{T}}\nabla_{\bm{x}^{\text{rg}}}(f^{\text{hp,sell,as}}-f^{\text{as,buy,hp}})= \bm{0}, \\
			&\partial{L}/\partial \bm{x}^\text{hp}= \bm{0},~\partial{L}/\partial \bm{x}^\text{as}= \bm{0}, \label{eq:L/xH}\\
			&(\bm{\lambda}^k)^\text{T} \bm{g}^k= \bm{0}, ~\bm{\lambda}^k\geq \bm{0}, ~\bm{g}^k\leq \bm{0},~\forall k \in \{\text{rg,hp,as}\}, \label{eq:complementary}\\
            &\bm{h}^k = \bm{0},\ \forall k \label{eq:equ}\\
			&\text{(\ref{eq:market1})},\  \text{(\ref{eq:market2})}%,~\rho^{\text{e}1},\rho^{\text{e}2},\rho^{\text{h}}\text{~are~free}
			\label{eq:marketall}
		\end{align}\label{eq:kkt}\end{subequations}}where $L$ is the Lagrangian of (\ref{eq:newproblem}); %(\ref{eq:L/xR}) is $\partial{L}/\partial \bm{x}^\text{rg}$ and (\ref{eq:L/xH}) is similar to (\ref{eq:L/xR}), which will not be further explanted here;
	(\ref{eq:L/xH}) are derived similarly to (\ref{eq:L/xR}), with the detailed expressions not given here due to space limitations;\textcolor{black}{
	$\nabla_{\bm{x}^{\text{rg}}}(C^{\text{hp}}+C^{\text{as}})$, $\nabla_{\bm{x}^{\text{rg}}}\bm{g}^{\text{hp}}$, $\nabla_{\bm{x}^{\text{rg}}}\bm{h}^{\text{hp}}$, $\nabla_{\bm{x}^{\text{rg}}}\bm{g}^{\text{as}}$, and $\nabla_{\bm{x}^{\text{rg}}}\bm{h}^{\text{as}}$} are all zero. As a result, (\ref{eq:kkt}) can be reformulated as:
	\begin{subequations}\color{black} \label{eq:kkt2}
		\begin{align}
			\nonumber & \nabla_{\bm{x}^{k}}C^{k} + (\bm{\lambda}^{k})^{\text{T}}\nabla_{\bm{x}^{k}} \bm{g}^{k}+(\bm{\mu}^{k})^{\text{T}}\nabla_{\bm{x}^{k}} \bm{h}^{k}\\
            \nonumber & + (\bm{\rho}^{\text{rg-hp,e}})^{\text{T}}\nabla_{\bm{x}^{k}}(P^{\text{rg,sell,hp}}-P^{\text{hp,buy,rg}}) \\
			\nonumber & + (\bm{\rho}^{\text{rg-as,e}})^{\text{T}} \nabla_{\bm{x}^{k}}(P^{\text{rg,sell,as}} - P^{\text{as,buy,rg}})\\
			& + (\bm{\rho}^{\text{hp-as,h}})^{\text{T}} \nabla_{\bm{x}^{k}}(f^{\text{hp,sell,as}}-f^{\text{as,buy,hp}})= \bm{0},~\forall k,\\
			& \text{(\ref{eq:complementary})},\ \text{(\ref{eq:equ})},\ \text{(\ref{eq:marketall}).}
		\end{align}
	\end{subequations}
Clearly, (\ref{eq:kkt2}) is equivalent to the system of equations $\big\{ \Omega^{\text{rg}}, \Omega^{\text{hp}}, \Omega^{\text{as}} \big\}$ under condition (\ref{eq:ve}). Thus, the optimal solution of (\ref{eq:newproblem}) is the same as the VE of the original model (\ref{eq:equilibrium}).
\end{proof}

 In this way, we can solve an SOCP to obtain the equilibrium rather than the nonconvex system of equations $\big\{ \Omega^{\text{rg}}, \Omega^{\text{hp}}, \Omega^{\text{as}} \big\}$.

\subsection{Temporal Decomposition-Based Solution Algorithm}
\label{sec:4.3gbd}

The proposed multistakeholder planning problem involves ASY and AST, thus requiring attention to long-term power and mass balance due to their slow dynamics \cite{li2022coplanning}. %, including monthly and quarterly considerations .
Therefore, the classical power system planning paradigms based on typical days are no longer applicable here. To address the issues of long-term balance, we consider a consecutive operational horizon composed of 12 typical weeks selected from 12 months.

 When considering an operational horizon of 12 weeks with a step length of 1 hour, there are 274,200 constraints and 175,403 variables. Although the equivalent MSSE problem (\ref{eq:newproblem}) is an SOCP, which is typically considered easy to solve, it is challenging to solve in this case because of its high dimensionality. % due to high-dimensional variables and constraints.

Fortunately, the horizon $T$ can be divided into several consecutive horizons of length $\tau$ with only a few temporally coupled constraints, such as the states of the HST and AST and the ramping of the ASY.
Consequently, (\ref{eq:newproblem}) can be decomposed into a planning master problem (MP) and 12 operational subproblems (SPs), as illustrated in Fig. \ref{fig:Tt}, which are easy to solve via the GBD.

Furthermore, the MP is a simple linear programming (LP). In GBD, adding multiple cuts, even if some are ineffective, almost does not slow down the solution of MP but rather accelerates the convergence compared to adding a single effective cut. Hence, we solve it using multi-cut GBD and parallelize the SPs to speed up the solving process.

\begin{figure}[t]
	\centering
	\includegraphics[width=3.4in]{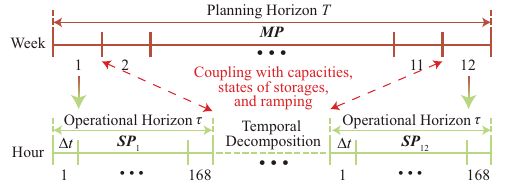}\vspace{-0 pt}
	\caption{Illustration of temporal decomposition in the planning and operational horizons.}
	\label{fig:Tt}\vspace{-0pt}
\end{figure}

For easy understanding, problem (\ref{eq:newproblem}) is rewritten as follows:
\begin{subequations}
	\begin{align}
		\min_{\{\bm{x}_{\text{inv}},\bm{x}_{\text{ope},l},\forall l\}}\ & \bm{a}^\text{T}\bm{x}_{\text{inv}} +  \sum\nolimits_{l=1}^{12} \bm{b}^\text{T}\bm{x}_{\text{ope},l},  \label{eq:globalobj}\\
		\text{s.t.}~~~~~~ &\bm{x}_{\text{inv}}\in X_{\text{inv}},\ \bm{x}_{\text{ope},l}\in X_{\text{ope},l} , \\
		& {\bm{C}}\bm{x}_{\text{inv}}+\bm{D}\bm{x}_{\text{ope},l}\leq \bm{d},~\forall l\in[1,...,12],    \label{eq:couple-st}
	\end{align}
\end{subequations}
where $\bm{x}_{\text{inv}}$ comprises the planning capacities and operational variables related to slow dynamics, including $S^{\text{hp/as,hst}}$, $S^{\text{as,ast}}$ and $P^{\text{as,asy}}$; $\bm{x}_{\text{ope},l}$ are all operational variables w.r.t. the $l$th week; their feasible sets are $X_{\text{inv}}$ and $X_{\text{ope},l}$, respectively; $\bm{a}$, $\bm{b}$, $\bm{d}$, $\bm{C}$, and $\bm{D}$ are the coefficient matrices; (\ref{eq:couple-st}) are the coupling constraints involving \textcircled{1} operational variables constrained by planning capacities in SPs; and \textcircled{2} the states of HST, AST, and ramping of ASY between adjacent SPs. The overall procedure of solving the MSSE via multicut GBD is summarized as follows.

\textit{Step 0:} Initialize $\bm{\hat{x}}_{\text{inv}}$, the iteration index $r\gets1$, and the convergence tolerance $\epsilon$; set the initial upper bound \textit{UB} and lower bound \textit{LB} of the objective (\ref{eq:globalobj}).

\textit{Step 1:} Solve the 12 parallelized SPs with fixed $\bm{\hat{x}}_{\text{inv}}$, with $l=1,\ldots,12$, as %The $l$-th SP follows
\begin{align}
	\nonumber \min_{\bm{x}_{\text{ope},l}}~& \bm{b}^{\text{T}}\bm{x}_{\text{ope},l} , \\
	\nonumber \text{s.t.} ~~&\bm{x}_{\text{ope},l}\in X_{\text{ope},l} , \\
	& \bm{C}\bm{\hat{x}}_{\text{inv}} + \bm{D}\bm{x}_{\text{ope},l}\leq \bm{d}: \bm{\mu}_{l}^{\text{opt}}.  \label{eq:sp}
\end{align}

\textit{Step 1a:} If the $l$th SP is feasible, obtain the optimal solution   $\bm{\hat{x}}_{\text{ope},l}$ %of (\ref{eq:sp})
and dual variables $\mu_{l}^{\text{opt}}$, and generate the optimal cut, as % for MP, as:
\begin{align}
	\text{LBD}_l\geq \bm{b}^\text{T}\hat{\bm{x}}_{\text{ope},l} +  ({\bm{C}}\boldsymbol{x}_\text{inv} + \bm{D}\hat{\bm{x}}_{\text{ope},l}-\bm{d})^{\text{T}} \bm{\mu}_l^\text{opt}. \label{eq:optcut}
\end{align}

\textit{Step 1b:}  If the $l$th SP is infeasible, construct a relaxed SP:
\begin{align}
	\nonumber \min_{\{\bm{x}_{\text{ope},l}, \bm{s}\}}~&\bm{1}^{\text{T}} \bm{s}, \\
	\nonumber \text{s.t.} ~~~~&\bm{x}_{\text{ope},l}\in X_{\text{ope},l}, \\
	& \bm{C}\hat{\bm{x}}_{\text{inv}} + \bm{D}\bm{x}_{\text{ope},l} - \bm{d} \leq \bm{s}: \bm{\mu}_{l}^{\text{fea}}, \label{eq:rsp}
\end{align}
where $\bm{s}$ is the introduced relaxation vector. Solve the relaxed SP (\ref{eq:rsp}) to obtain the optimal solution $\bm{\hat{x}}_{\text{ope},l}$ and dual variable $\bm{\mu}_{l}^{\text{fea}}$ and generate a feasible cut
\begin{align}
	(\bm{C}\bm{x}_{\text{inv}}+\bm{D}\hat{\bm{x}}_{\text{ope},l} - \bm{d})^{\text{T}}   \bm{\mu}_{l}^{\text{fea}}  \leq 0.\label{eq:feacut}
\end{align}

\textit{Step 1c:} If all the SPs are feasible, then update the upper bound $\textit{UB} \leftarrow \min\ \{\textit{UB}, \bm{a}^\text{T}\bm{\hat{x}}_\text{inv}+\sum_{l=1}^{12} \bm{b}^\text{T}\bm{\hat{x}}_{\text{ope},l}\}$.

\textit{Step 2:} Update $\bm{\hat{x}}_{\text{inv}}$ by solving the MP
\begin{align}
	\hspace{-3pt} \nonumber\min_{\{\bm{x}_{\text{inv}},\text{LBD}_{l}, \forall l\}}~ & {\bm{a}}^{\text{T}}\bm{x}_{\text{inv}} + \sum\nolimits_{l=1}^{12}\text{LBD}_{l}, \\
	\nonumber \text{s.t.}~~~~~~~ &\bm{x}_{\text{inv}}\in X_{\text{inv}}, \\
	& \text{all optimal cuts (\ref{eq:optcut}) and feasible cuts (\ref{eq:feacut})}, \label{eq:mp}
\end{align}
where $\text{LBD}_{l}$ is the lower bound of the augmented Lagrangian of the $l$-th SP. Then, update $\textit{LB}\leftarrow\bm{a}^\text{T}\hat{\bm{x}}_\text{inv}+\sum_{l=1}^{12}\text{LBD}_l$.

\textit{Step 3:} If $|(\textit{UB}-\textit{LB})/\textit{LB}|\leq \epsilon$, output the obtained optimum; otherwise, update $r \gets r + 1$ and return to \textit{Step 1}.

For more details of the multicut GBD algorithm, interested readers are referred to \cite{geoffrion1972generalized,birge1988multicut}. We do not % explore this
go  further here.

\section{Case Studies}
\label{sec:cases}

\subsection{Case Setups}
\label{sec:Systems}

We analyze the MSSE of an off-grid ReP2A system based on a real-life project in North China, as shown in Fig. \ref{fig:case}. The system includes a 300-MW wind farm and a 100-MW photovoltaic (PV) array. Wind and solar data from 12 typical weeks, selected from historical records \cite{yu2023optimal}, as shown in Fig. \ref{fig:WS}. Their FLHs are about 3,033 and 1,808 hours, respectively. The external ammonia price is based on the historical market price of fossil fuel-based ammonia \cite{wu2023multi}. Key investment and operation parameters are detailed in Tables \ref{tab:para-inv} and \ref{tab:para-ope}, respectively.

\begin{table}[tb]\footnotesize
	\renewcommand{\arraystretch}{1.35}
	\caption{The Key Investment Parameters for Case Studies}\vspace{-0pt}
	\label{tab:para-inv}
	\centering
	\begin{tabular}{cccc}
		\hline\hline
		Facility                           & Unit investment cost                  & Lifetime $y$ (years)             \\
		\hline
		WT                                 & 5,000 CNY/kW                            & 20           \\
		PV                                 & 4,000 CNY/kW                             & 20            \\
		BES                                & 1,500 CNY/kWh                            & 20                   \\
		VC                                 & 200  CNY/kVar                           & 15                          \\
		AE                                 & 3,500 CNY/kW                             & 10                \\
		HST                                & 250  CNY/Nm$^3$                         & 20               \\
		ASY                                & 21,706 CNY/(t/h)                            & 10                \\
		AST                                & 3300 CNY/t                              & 20                   \\
		Transmission line                  & 2,000,000 CNY/km                              & 40                   \\
		Hydrogen pipeline                  & 4,000,000 CNY/km                             & 40                   \\
		\hline\hline
	\end{tabular}
	\vspace{10pt}
	\renewcommand{\arraystretch}{1.5}
	\caption{Parts of the Operation Parameters for Case Studies}\vspace{-0pt}
	\label{tab:para-ope}
	\centering
	\begin{tabular}{cccc}
		\hline\hline
		Symbol                                                            & Value                 & Symbol                                                    &Value            \\
		\hline
		$T/\tau / \Delta t$                                               & 2016/168/1 h          & $\eta^{\text{bes,c}}/\eta^{\text{bes,d}}$                     &90\%/95\%       \\
		$\overline{\eta}^{\text{bes}}/\underline{\eta}^{\text{bes}}$      & 90\%/10\%            &$\overline{\upsilon}_{i}/\underline{\upsilon}_{i}$         & 1.05$^2$/0.95$^2$             \\
		$\overline{\eta}^{\text{ae}}/ \underline{\eta}^{\text{ae}}$       & 100\%/5\%             & $\overline{\eta}^{\text{hst}}/\underline{\eta}^{\text{hst}}$  & 100\%/10\%                 \\
		$\eta^{\text{p2h}}$                                               & 0.2 Nm$^3$/kWh \cite{yu2024optimal}        &$\eta^{\text{hp,comp}}$                                     & 0.033 kWh/Nm$^3$ \cite{yu2024optimal}           \\
		$\eta^{\text{h2a}}$                                               & 0.5027 kg/Nm$^3$ \cite{wu2023multi}     &$\eta^{\text{p2a}}$                                        & 1.5664 kg/kWh  \cite{wu2023multi}          \\
		$\overline{\eta}^{\text{asy}}/\underline{\eta}^{\text{asy}}$        & 100\%/30\%  \cite{wu2023multi}          & $\overline{r}^{\text{asy}}/\underline{r}^{\text{asy}}$      & 20\%/20\% \cite{yu2023optimal,yu2024optimal}                  \\
		\hline\hline
	\end{tabular}\vspace{-0pt}
\end{table}

\begin{figure}[tb]
	\centering
	\includegraphics[width=3.45in]{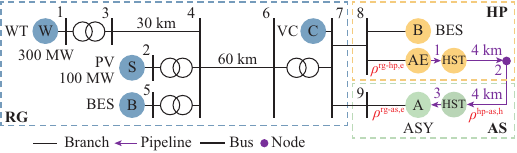}\vspace{-0pt}
	\caption{Topology of the off-grid ReP2A system used in the case study.}
	\label{fig:case}\vspace{-0pt}
\end{figure}

 The simulations are performed via \emph{Wolfram Mathematica 14.0} on a desktop with an \emph{Intel Core i5-12400@2.5 GHz} CPU and 16 GB of RAM. The MP and SPs in the temporal decomposition-based algorithm are solved via \emph{Gurobi 11.0} and \emph{Mosek 10.1}, respectively. The convergence tolerance $\epsilon$ is set at 10$^{-4}$. The overall solving time is approximately 40 minutes.

To account for the interests of the three stakeholders, we define the average transaction prices of electricity and hydrogen $\rho_{\text{avg}}^{\text{rg-hp,e}}$, $\rho_{\text{avg}}^{\text{rg-as,e}}$, and $\rho_{\text{avg}}^{\text{hp-as,h}}$ as
\begin{align}
	\rho_{\text{avg}}^{\text{rg-hp,e}} &= \sum\nolimits_t \rho_t^{\text{rg-hp,e}} P_t^{\text{hp,buy,rg}}/\sum\nolimits_t P_t^{\text{hp,buy,rg}}, \\
	\rho_{\text{avg}}^{\text{rg-as,e}} &= \sum\nolimits_t \rho_t^{\text{rg-as,e}} P_t^{\text{as,buy,rg}}/\sum\nolimits_t P_t^{\text{as,buy,rg}}, \\
	\rho_{\text{avg}}^{\text{hp-as,h}} &= \sum\nolimits_t \rho_t^{\text{hp-as,h}} f_t^{\text{as,buy,hp}}/\sum\nolimits_t f_t^{\text{as,buy,hp}}.
\end{align}
Furthermore, we use the LCOA to evaluate the overall technoeconomic performance of the ReP2A system as follows:
\begin{align}
	& \text{LCOA} = (\sum_k C^k-\dfrac{8760}{T}\sum_t M_t^{\text{as,sell}}\rho_t^{\text{a}})/(\dfrac{8760}{T}\sum_t M_t^{\text{as,pro}}).
\end{align}
\vspace{-16pt}

\begin{figure}[t]
	\centering\vspace{-0pt}
	\includegraphics[width=3.45in]{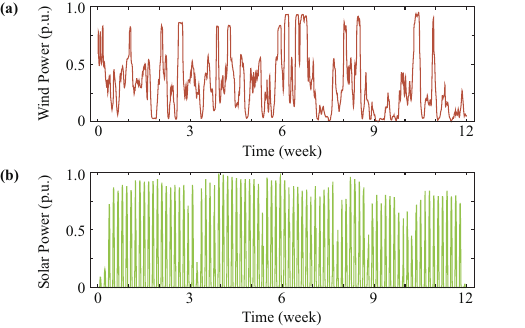}\vspace{0pt}
	\caption{Wind and solar power data of typical weeks used in the case study.}
	\label{fig:WS}\vspace{0pt}
\end{figure}

\begin{figure}[t]
	\centering
	\includegraphics[width=3.5in]{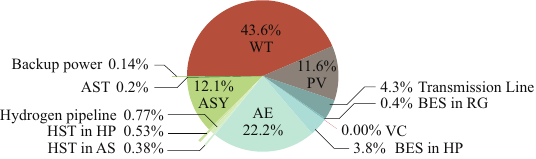}\vspace{0pt}
	\caption{Composition of the LCOA in the base case.}
	\label{fig:bar}
\end{figure}

\begin{figure}[t]
	\centering
	\includegraphics[width=3.45in]{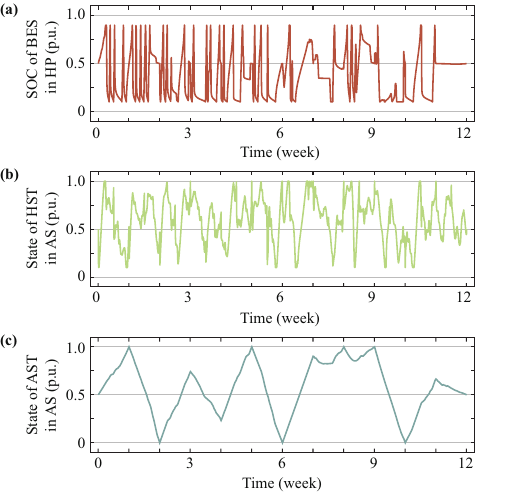}\vspace{0pt}
	\caption{The states of the BES in the HP, the HST in the AS, and the AST in the AS across the 12-week horizon. (a) BES. (b) HST. (c) AST.}
	\label{fig:ES}\vspace{-0pt}
\end{figure}

\begin{figure}[t]
	\centering\vspace{-0pt}
	\includegraphics[width=3.45in]{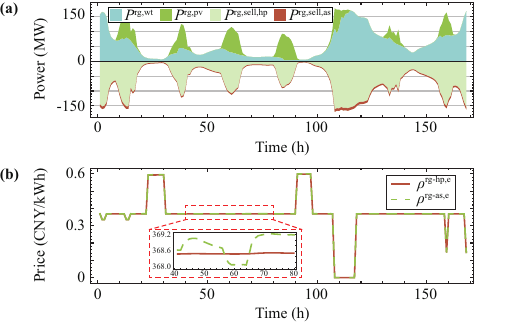}\vspace{0pt}
	\caption{Renewable power generation, load and equilibrium electricity prices. (a) Renewable power and loads of the AE and ASY. (b) Electricity prices. }
	\label{fig:powerprice}\vspace{12pt}
	\includegraphics[width=3.45in]{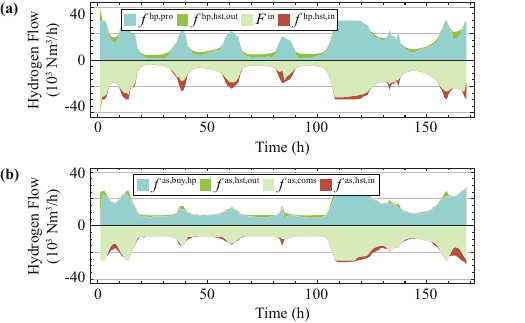}\vspace{0pt}
	\caption{Hydrogen balance at hydrogen nodes. (a) HP node. (b) AS node.}
	\label{fig:hydrogen}
\end{figure}

\subsection{Analysis of Typical Operational and Trading Results under Equilibrium in the Base Case}
\label{sec:Equilibrium}

We analyze the operational behaviors of the staged BSs and conduct a base-case equilibrium analysis. The planning result obtained by the MSSE model is: $W^{\text{rg/hp,bes}}=7.29/77.13$ MWh, $W^{\text{rg,vc}}=0.01$ MVar, $W^{\text{hp,ae}}=149.13$ MW, $W^{\text{hp/as,hst}}=52/72.42\times10^3$ Nm$^3$, $W^{\text{as,asy}}=13.1$ t/h, and $W^{\text{as,ast}}=2,430.76$ t, as listed in Table \ref{tab:comparison}, labeled as {\bf C1}. The LCOA composition is shown in Fig. \ref{fig:bar}. Fig. \ref{fig:ES} illustrates the states of BES in HP, HST in AS, and AST in AS throughout the planning horizon. The state of the BES in RG and the state of the HST in HP are not listed here due to the space limit.
Fig. \ref{fig:powerprice}(a) shows the wind and solar power and loads of AE and ASY in the 8th week, with the electricity prices presented in Fig. \ref{fig:powerprice}(b). Fig. \ref{fig:hydrogen} illustrates the hydrogen flow balance at the HP and AS nodes. Fig. \ref{fig:amprice} shows the average electricity and hydrogen prices of AS and the ammonia price each week.

As shown in Fig. \ref{fig:ES}, the fluctuating patterns of the BES, HST, and AST are daily, weekly, and cross-weekly, respectively, each suited to different time scales. The states of the BES and HST fluctuate with wind and solar power, while the state of the AST follows the opposite trend of the ammonia price, as shown in Fig. \ref{fig:amprice}, reflecting its profit-chasing behavior.

Fig. \ref{fig:powerprice} shows that peak, flat, and valley periods of wind and solar power correspond to low, medium, and high electricity prices under equilibrium, such as $t\in[108,117]$ h, $t\in[98,107]$ h, and $t\in[91,97]$ h, respectively. During flat periods, the prices stabilize around $0.37$ CNY/kWh. Microscopically, $\rho^{\text{rg-as,e}}$ depends more on wind and solar resources compared to $\rho^{\text{rg-hp,e}}$ due to the smaller invested capacity of the RG-side BES.
During high periods, the electricity price approaches $0$ CNY/kWh due to the saturated production capacities of AE and ASY, causing the renewable power supply to exceed the demand. During low periods, the prices stay around $0.6$ CNY/kWh, as it is enveloped by the backup power price $\rho^{\text{as,back}}$ ($0.6$ CNY/kWh) for the ASY.

Combining Figs. \ref{fig:ES} and \ref{fig:powerprice}, we observe that the fluctuations of renewable generation, hydrogen production, and hydrogen consumption (i.e., ASY) are progressively buffered by the staged BESs and HSTs. The hydrogen trading quantities follow the electricity trading. Further examining Fig. \ref{fig:amprice}, we can find that the trend of the hydrogen prices is highly consistent with the electricity prices and shows some consistency with the external ammonia price. Since wind and solar resources are negatively correlated with electricity prices, we infer a negative correlation between the renewables and the equilibrium hydrogen price.

Moreover, we can conclude that the cost of green ammonia, influenced by the electricity and hydrogen prices, is significantly affected by the uncertainty of wind and solar resources. This may pose risks to both the producers and consumers in the green ammonia market. To further enhance the competitiveness of green ammonia, improving trading mechanisms by introducing spot markets \cite{an2024decomposition} or financial tools \cite{wu2023multi,cai2023green} is recommended.

\begin{figure}[t]
	\centering
	\includegraphics[width=3.5in]{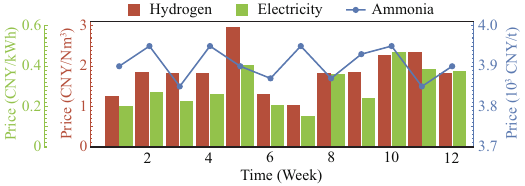}\vspace{-0pt}
	\caption{Electricity, hydrogen, and external ammonia prices over 12 weeks.}
	\label{fig:amprice}\vspace{-0pt}
\end{figure}

\label{sec:planning results}
\begin{table*}[tb]\scriptsize
	\renewcommand{\arraystretch}{1.8}
	\caption{Planning Results and Performance Metrics of Different Cases}\vspace{0pt}
	\label{tab:comparison}
	\centering
	\setlength{\tabcolsep}{4.5pt}
	\begin{tabular}{cccccc}
		\hline\hline
		\vspace{-10pt}
		\\
		Case       &\tabincell{c}{Planning capacities\\$\{W^{\text{rg,bes}},W^{\text{rg,vc}}, W^{\text{hp,bes}},W^{\text{hp,ae}},W^{\text{hp,hst}}, W^{\text{as,hst}},W^{\text{as,asy}},W^{\text{as,ast}}\}$\\ (MWh, MVar, MWh, MW, Nm$^3$, Nm$^\text{3}$, t/h, t)}    & \tabincell{c}{Profits of RG, HP and AS\\\{$-C^{\text{rg}},-C^{\text{hp}},-C^{\text{as}}$\}\\
			(million CNY)}  & \tabincell{c}{Social welfare\\$-\textstyle\sum_k C^k$\\(million CNY)}  & \tabincell{c}{\textcolor{black}{LCOA}\\(CNY/t)} & \tabincell{c}{Trading prices\\\{$\rho_{\text{avg}}^{\text{rg-hp,e}},\rho_{\text{avg}}^{\text{rg-as,e}},\rho_{\text{avg}}^{\text{hp-as,h}}$\}\\(CNY/kWh, CNY/kWh, CNY/Nm$^3$)} \\
		\hline
		\textbf{C1}              & \{7.29, 0.01, 77.13, 149.13, 52.00$\times$10$^3$, 72.42$\times$10$^3$, 13.10, 2430.76\}           &\{\textendash11.60, 5.35, \textendash24.15\}      &\textendash30.40 &4298.5 & \{0.2288, 0.2445, 1.765\}   \\
		
		\textbf{C2}              & \{\textbf{0}, 0.05, 84.55, 149.29, 48.35$\times$10$^3$, 77.82$\times$10$^3$, 13.08, 2427.59\}        &\{\textendash36.89, \textendash3.35, 9.84\}  &\textendash30.40     &4298.3    &\{0.1983, 0.2181, 1.568\}   \\
		
		\textbf{C3}              & \{85.67, 0.11, \textbf{0}, 150.62, 45.21$\times$10$^3$, 92.91$\times$10$^3$, 13.08, 2410.34\}        &\{\textendash25.03, 1,69, \textendash7.14\}  &\textendash30.48     &4299.1    &\{0.2298, 0.2305, 1.665\}   \\
		
		\textbf{C4}              & \{14.75, 0.02, 70.18, 149.754, \textbf{0}, 130.34$\times$10$^3$, 13.17, 2419.91\}       &\{\textendash21.53, 23.80, \textendash32.83\}  &\textendash30.83     &4303.7    &\{0.2187, 0.2368, 1.810\}   \\
		
		\textbf{C5}              & \{12.52, 0.19, 73.18, 150.74, 130.65$\times$10$^3$, \textbf{0}, 13.22, 2425.50\}       &\{\textendash38.70, 5.73, 2.32\}  &\textendash30.98     &4301.0    &\{0.1983, 0.2152, 1.625\}   \\
		
		\textbf{C6}             & \{85.44, 0.01, \textbf{0}, 150.33, 131.32$\times$10$^3$, \textbf{0}, 13.19, 2426.35\}       &\{\textendash38.30, 24.31, \textendash16.75\}  &\textendash30.74     &4302.2    &\{0.2139, 0.2314, 1.735\}   \\
		\hline
		\textbf{C7}             & \tabincell{c}{$W^{\text{rg,wt}}=$ 300 MW, $W^{\text{rg,pv}}=$ 113.35 MW\\\{48.78, 0.06, 48.60, 165.2, 186.60$\times$10$^3$, 11.63$\times$10$^3$, \textbf{15.7}, 2276.96\}}     &\{\textendash4.29, 36.28, \textendash68.97\}  &\textendash36.98     &4348.7    &\{0.2377, 0.2407, 1.987\}   \\
		\hline
		\textbf{C8}             & \tabincell{c}{$W^{\text{rg,wt}}=\textbf{200}~\text{MW}, W^{\text{rg,pv}}=\textbf{260}~\text{MW}$\\\{\textbf{92}, \textbf{0}, \textbf{0}, \textbf{125}, \textbf{100$\times$10$^3$}, \textbf{0}, \textbf{15.7}, \textbf{3000}\}}     &\{\textendash111.69, 92.84, \textendash66.03\}  &\textendash84.88     &5108.5    &\{0.1761, 0.1619, 2.020\}   \\
\hline
		\color{black}\textbf{C9}             & \color{black}\{0, 0\}, \{0, 100, 0\}, \{0, 10, 1512.00\}     &\color{black}\{\textendash10.33, \textendash53.10, \textendash54.84\}  &\color{black}\textendash118.27     &\color{black}6034.66    &\color{black}/   \\
		\hline
		\color{black}\textbf{C10}             & \color{black}\tabincell{c}{\{16.16, 0, 68.87, 149.90, null, null, 13.18, 2422.94\}\\$W^{\text{h\&p,hst}}=$132.49$\times$10$^3$ Nm$^3$}     &\color{black} \tabincell{c} { \{\textendash44.60, null, null\} \\ $-C^{\text{hp}}-C^{\text{as}}$ = 13.69}  &\color{black}\textendash30.91     &\color{black}4307.35   &\color{black}$\rho^{\text{rg-h\&a,e}}=$ 0.2063 CNY/kWh   \\
		\hline
		%\textbf{D1}             & \{10.53, 0.10, 74.11, 149.41, 44.35$\times$10$^3$, 80.89$\times$10$^3$,13.10, 2430.1\}     &\{9.16, \textendash5.10, 8.23\}  &\textendash12.29     &4080.4    &\{0.2144, 0.2336, 1.671\}   \\
		%\hline
		\textbf{D1}             & \{28.59, 0.06, 62.26, 157.22, 61.36$\times$10$^3$, 83.25$\times$10$^3$, 13.58, 2380.3\}     &\{6.34, 14.96, \textendash17.90\}  &3.40     &4292.6   &\{0.2476, 0.2711, 1.913\}   \\
		\textbf{D2}             & \{\textbf{0}, 0.02, 90.83, 157.22, 67.53$\times$10$^3$, 78.18$\times$10$^3$, 13.58, 2384.6\}     &\{\textendash5.27, 0.66, 8.01\}  &3.40     &4293.6    &\{0.2289, 0.2510, 1.767\}   \\
		\textbf{D3}             & \{90.99, 0.04, \textbf{0}, 157.31, 65.89$\times$10$^3$, 81.99$\times$10$^3$, 13.57, 2383.9\}     &\{\textendash6.26, 11.63, \textendash2.05\}  &3.32     &4293.6    &\{0.2466, 0.2610, 1.823\}   \\
		\textbf{D4}             & \{13.01, 0.02, 73.10, 151.23, \textbf{0}, 136.41$\times$10$^3$, 13.28, 2412.5\}     &\{11.42, 9.22, \textendash18.03\}  &2.61     &4298.3    &\{0.2557, 0.2545, 1.914\}   \\
		\textbf{D5}             & \{30.34, 0.00, 61.02, 157.83, 153.75$\times$10$^3$, \textbf{0}, 13.57, 2383.9\}     &\{7.23, 10.56, \textendash14.65\}  &3.14     &4296.0    &\{0.2491, 0.2620, 1.908\}   \\
		\textbf{D6}             & \{91.50, 0.04, \textbf{0}, 157.31, 152.46$\times$10$^3$, \textbf{0}, 13.66, 2380.3\}     &\{\textendash17.64, 10.03, 10.68\}  &3.07     &4296.8    &\{0.2339, 0.2452, 1.763\}   \\
		\hline
		\textbf{E1}             & \{28.59, 0.06, 62.26, 157.22, 61.36$\times$10$^3$, 83.25$\times$10$^3$, 13.58, 2380.3\}     &\{0.05, 1.82, 1.53\}  &3.40     &4292.6   &\{0.2476, 0.2711, 1.913\}   \\
        \textcolor{black}{\textbf{E2}}             & \textcolor{black}{\{28.59, 0.06, 62.26, 157.22, 61.36$\times$10$^3$, 83.25$\times$10$^3$, 13.58, 2380.3\}}     &\textcolor{black}{\{1.83, 1.22, 0.35\}}  &\textcolor{black}{3.40}     &\textcolor{black}{4292.6}   &\textcolor{black}{/}   \\
        \textcolor{black}{\textbf{E3}}             & \textcolor{black}{\{28.59, 0.06, 62.26, 157.22, 61.36$\times$10$^3$, 83.25$\times$10$^3$, 13.58, 2380.3\}}     &\textcolor{black}{\{1.84, 1.02, 0.54\}}  &\textcolor{black}{3.40}     &\textcolor{black}{4292.6}   & \textcolor{black}{\{0.2426, 0.2657, 1.806\}}  \\
        \hline
		\textcolor{black}{\textbf{F1}}             & \textcolor{black}{\{36.19, 0.25, 57.96, 161.07, 127.64$\times$10$^3$, 37.88$\times$10$^3$, 14.17, 2338.7\}}     &\color{black}\{\textendash27.35, 22.93, \textendash22.96\}  &\color{black}\textendash45.16     &\color{black}4228.0   &\color{black}\{0.2056, 0.2284, 1.747\}   \\
		\hline\hline
	\end{tabular}\vspace{0pt}
\end{table*}

\subsection{Planning Results for the Multistakeholder Sizing Equilibrium under Different Configurations}
\label{sec:planning}

We further analyze the MSSE under \textcolor{black}{ten} cases, \textcolor{black}{examining the differences in sizing results under various configurations, to reveal a potential dilemma in multistakeholder investment practices during the implementation of ReP2A systems}:
\begin{itemize}
	\item \textbf{C1} (The base case in Section \ref{sec:Equilibrium}): the RG and HP both invest in BESs, and the HP and AS both invest in HSTs. In other words, $W^{\text{rg,bes}}$, $W^{\text{hp,bes}}$, $W^{\text{hp,hst}}$, and $W^{\text{as,hst}}$ all must be optimized, as described in Sections \ref{sec:structure} and \ref{sec:model}.
	\item \textbf{C2}: The RG is not allowed to configure BES ($W^{\text{rg,bes}}=0$).
	\item \textbf{C3}: The HP is not allowed to configure BES ($W^{\text{hp,bes}}=0$).
	\item \textbf{C4}: The HP is not allowed to configure HST ($W^{\text{hp,hst}}=0$).
	\item \textbf{C5}: The AS is not allowed to configure HST ($W^{\text{as,hst}}=0$).
	\item \textbf{C6}: The HP is not allowed to configure BES, and the AS is not permitted to configure HSTs ($W^{\text{hp,bes}}=0$, $W^{\text{as,hst}}=0$).
	\item \textbf{C7}: The WT and PV capacities are not fixed, i.e., $W^{\text{rg,wt}}$ and $W^{\text{rg,pv}}$ are to be optimized, and the capacity of ASY is fixed at $W^{\text{as,asy}}=15.7$ t/h, as in \cite{yu2023optimal}. The other settings are consistent with \textbf{C1}.
	\item \textbf{C8}: The capacities of all facilities are fixed at $W^{\text{rg,wt}}=200$ MW, $W^{\text{rg,pv}}=260$ MW, $W^{\text{rg,bes}}=92$ MWh, $W^{\text{rg,vc}}=0$ MVar, $W^{\text{hp,bes}}=0$ MW, $W^{\text{hp,ae}}=125$ MW, $W^{\text{hp,hst}}=10^5$ Nm$^3$, $W^{\text{hp,hst}}=0$ Nm$^3$, $W^{\text{as,asy}}=15.7$ t/h and $W^{\text{as,ast}}=3000$ t, the same as in \cite{yu2023optimal, wu2023multi}. The other settings are consistent with \textbf{C1}.
    \item \color{black}\textbf{C9}: RG, HP, and AS plan separately without integration. RG generates electricity for the power grid (price = 0.18 CNY/kWh), HP buys electricity at time-of-use price \cite{yu2023optimal} to produce hydrogen for the hydrogen market (selling price = 2 CNY/Nm$^3$), and AS buys hydrogen from hydrogen market (purchase price = 2.2 CNY/Nm$^3$) to produce ammonia.
    \item \color{black}\textbf{C10}: HP and AS are invested in by a single entity, represented as H\&A, without hydrogen trading process. The electricity price between RG and H\&A is defined as $\rho^{\text{rg-h\&a,e}}$. The HST is located on the ASY side and its capacity in the H\&A is defined as $W^{\text{h\&a,hst}}$.
\end{itemize}

Table \ref{tab:comparison} shows the planning results and operational performance.
In the base case (C1), the AS has the highest cost among the three stakeholders, with almost all profits transferred to the HP, which aligns with the AS's reluctance to participate in a demo ReP2A project in North China.
Comparing C1 and C2, we can see that the investment in BES by the RG can significantly increase electricity prices, as explained in Section \ref{sec:Equilibrium} and Fig. \ref{fig:powerprice}. Therefore, the current policies requiring integrated investment in wind and solar power and  BES benefit the interests of the RG in implementing ReP2A projects.
Further analysis of C1 and C3 reveals that investing in the BES by the HP also increases its benefits by increasing flexibility in response to fluctuations in wind, solar, and electricity prices.

In case C7, where the ASY capacity is predetermined on the basis of ammonia demand, the overall technoeconomic viability of ReP2A is inferior to that in cases C1 to C6. This suggests that sizing on the basis of wind and solar resources is marginally more advantageous than sizing by ammonia demand. \textcolor{black}{It is important to note that even when the RG can optimize wind, solar, and BES capacities, its profit remains negative due to the fixed transmission line investment cost ($(1+\eta_{\text{O\&M}})\text{CRF}(r,y) C^\text{rg,line}$ = 15.40 million CNY). In reality, optimizing the equipment capacity generates a profit of 11.11 million CNY, resulting in a net profit of -4.29 million CNY.} In C8, with all facilities' capacities fixed, the overall technoeconomic viability of ReP2A becomes unoptimistic, and significant differences in stakeholders' interests can be observed. Therefore, we observe that coordinated sizing for renewable, AE, and ASY facilities, although conducted by stakeholders with conflicting interests, contributes to the technoeconomic viability of ReP2A.

\textcolor{black}{Under case C9, there is no electricity or hydrogen trading and delivery process among the three entities. RG does not deploy BES because the electricity price is fixed. Similarly, HP and AS do not deploy HST. HP does not deploy BES because using off-peak electricity for hydrogen production still cannot offset the BES costs. AS still configures AST to sell ammonia at higher prices in the external market. The sizing results of AE and ASY are both at their lower bounds (zero if no lower bounds exist), indicating poor technoeconomic performance and no appeal to any entity. Additionally, hydrogen production from carbon-emitting electricity for ammonia synthesis is ineligible for green subsidies \cite{del2022techno} and lacks application potential.}

\textcolor{black}{A comparison of C10 with C1 and C4 shows that the overall sizing results and technoeconomic differences of the ReP2A system between the integrated H\&A model and the three-entity model are minimal. However, the transactions and conflicts of interest between the entities differ significantly. The flexibility of the H\&A is improved compared to the individual HP and AS, leading to a reduction in marginal production cost when participating in electricity trading. As a result, the revenue of RG is transferred to the H\&A side. Therefore, encouraging companies to invest in both hydrogen and ammonia simultaneously can increase overall profits and avoid conflicts, compared to separate investments.  Additionally, independent HP and AS stakeholders are incentivized to form profitable collaborations.}

In summary, for cases C1 to C6, to improve the overall technoeconomic performance of the ReP2A system, the capacities of renewables, AE, ASY, and BSs must be properly aligned. However, despite efforts to optimize BES and HST configurations, stakeholders' interests remain unbalanced due to their heterogeneous characteristics. Specifically, in no case do all three stakeholders achieve a positive profit. This phenomenon is very different from the well-studied applications in power systems, for example, games among power sources \cite{Ruiz2012Equilibria}, microgrids \cite{naebi2020epec}, and prosumers \cite{wang2021distribute}. Analyzing \textcolor{black}{the} HP's cost %$C^{\text{hp}}$
in C1 to C6, we can see that the HP plays a dominant role in the equilibrium. To make ReP2A projects invested in by multiple stakeholders implementable, benefit transfer mechanisms or pricing agreements could be established before implementing a ReP2A project to ensure the interests of all parties.

\subsection{Recommendation for the Investment of ReP2A Projects}
\label{sec:extension}

According to the results from C1 to C6, %green ammonia's competition in the conventional ammonia market causes
at least one stakeholder will incur a loss, thus diminishing the overall motivation for investment. %and making the overall ReP2A project unimplementable.
To address this issue, providing subsidies \cite{del2022techno,zhao2022potential} for green ammonia
could increase economic competitiveness and stimulate the growth of the green %hydrogen %and
ammonia
industries. To explore this, we introduce six additional cases, denoted as \textbf{D1} to \textbf{D6}, where the ammonia price is increased by 400 CNY/t while other settings remain the same as those in C1--C6. The simulation results are shown in Table \ref{tab:comparison}.

In cases D1 to D6, although social welfare is positive, no one achieves mutual benefits under equilibrium, which is consistent with the findings from C1 to C6. Therefore, we conclude that the multistakeholder ReP2A system can hardly achieve reasonable investment under free competition. To address this, it is recommended that regulatory authorities establish a negotiation platform for multistakeholder investment. This would allow stakeholders to form a consortium \cite{yu2023optimal} and enter into benefit transfer or pricing agreements that ensure mutual benefits and facilitate successfully implementing ReP2A projects. \textcolor{black}{Possible methods include:}

\subsubsection{\textcolor{black}{Revenue Transfer}}
As an example, in case D1, RG and HP are profitable, whereas AS incurs a loss. To mitigate this, we propose a benefit transfer mechanism whereby RG transfers 3\% of its revenue from electricity sales \textcolor{black}{(i.e., $P^{\text{rg,sell,hp}}\rho^{\text{rg-hp,e}}$)} to the HP and where the HP transfers 6\% of its revenue from hydrogen sales \textcolor{black}{(i.e., $f^{\text{hp,sell,as}}\rho^{\text{hp-as,h}}$)} to the AS. \textcolor{black}{Based on the results of D1, transferring revenue from HP to AS alone cannot compensate for AS's loss. Therefore, RG also needs to participate in the benefit transfer.} The simulation results for this scenario, denoted as \textbf{E1}, are recorded in Table \ref{tab:comparison}, which shows that all three stakeholders achieve a positive profit. % As shown, multi-stakeholder investment in ReP2A projects can achieve mutual benefits under equilibrium through reasonable benefit transfer agreements.
Due to the space limit, we will not discuss this further.

\subsubsection{\textcolor{black}{Profit Re-arrangement}}
\textcolor{black}{We divide the total profit into a larger basic part and a smaller compensation part, and design a profit re-arrangement method. The basic part is arranged according to the investment cost proportions of the three parties, i.e., composition of the LCOA, ensuring that all parties receive positive base returns. The compensation part serves as an incentive for the beneficiaries (RG and HP here) under the equilibrium, and is allocated according to the proportion of their respective profits. Similarly, in Case D1, the total profit is divided according to an 8:2 ratio, with the basic part arranged approximately in a 6.0:2.7:1.3 ratio (as shown in Fig. \ref{fig:bar}) to RG/HP/AS. For example, RG's profit is calculated as $3.40\times0.8\times\frac{6.0}{6.0+2.7+1.3}+3.40\times0.2\times\frac{6.34}{6.34+14.96}\approx 1.83$ million CNY. The resulting arrangement is denoted as \textbf{E2}, as shown in Table \ref{tab:comparison}.}

\subsubsection{\textcolor{black}{Price Contacts}}
\textcolor{black}{In addition, multiple stakeholders can sign contracts for transaction prices, making adjustments based on marginal production costs. In Case D1, RG achieves positive profits at marginal electricity prices of approximately 0.2407 CNY/kWh and 0.2635 CNY/kWh, while AS achieves positive profits at a marginal hydrogen price of around 1.810 CNY/Nm$^3$. Based on the equilibrium prices, AS is evidently at a loss. Therefore, the hydrogen price can first be reduced based on the spot price $\rho^{\text{hp-as,h}}$ to ensure AS's profitability. If HP's profit becomes positive, pricing adjustments end. Otherwise, the electricity price is reduced to transfer part of RG's operational profit to HP. Here, the three parties can sign a contract to trade at electricity and hydrogen prices $\hat{\rho}^{\text{rg-hp/as,e}}$, and $\hat{\rho}^{\text{hp-as,h}}$, which are 0.98 times and 0.944 times of the equilibrium values \textcolor{black}{($\hat{\rho}^{\text{rg-hp/as,e}}=0.98\rho^{\text{rg-hp/as,e}}$  and $\hat{\rho}^{\text{hp-as,h}}=0.944\rho^{\text{hp-as,h}}$)}, respectively, denoted as \textbf{E3} in Table \ref{tab:comparison}.}

\textcolor{black}{The methods above can be designed based on the application, and the parameters (such as 3\%, 6\%, 8:2, 0.98, and 0.944) should be tailored, aiming to create conditions for mutual benefits among all stakeholders.}

\subsection{\color{black}Computational Performance}
\label{sec:compute}

\begin{figure}[t]
	\centering
	\includegraphics[width=3.45in]{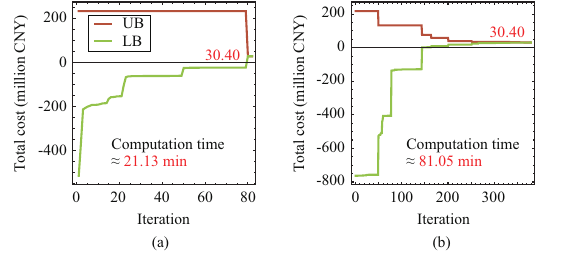}\vspace{-0pt}
	\caption{\textcolor{black}{Convergence and computation time of the proposed multicut and conventional BD algorithms. (a) Proposed. (b) Conventional.}}
	\label{fig:BDconvergence}\vspace{-0pt}
\end{figure}

\textcolor{black}{In the computational performance, we compare the proposed multi-cut Benders decomposition (BD) algorithm with the conventional single-cut BD algorithm, with convergence and computation time shown in Fig. \ref{fig:BDconvergence}. It can be observed that the proposed method significantly improves convergence and reduces computation time to approximately 21.13/81.05$\approx$ 1/4 compared to the traditional method, addressing the long-term planning problem in this work. Additionally, the solution time for each MP is less than 1 second, while each SP takes approximately 4.1 seconds. If the SPs could be fully solved in parallel, the solution time of each iteration would not exceed 10 seconds, keeping the total computation time within 15 minutes for about 82 iterations.}

\textcolor{black}{This paper focuses on the planning of ReP2A but does not involve the selection of power lines, pipelines, and electrolyzers. As the number of planning objects and the dimensions of decision variables and constraints increase, long-term planning problems will become more complex, highlighting the need for more efficient decomposition algorithms.}

\subsection{\color{black}Impact of Uncertainty on Sizing Equilibrium}
\textcolor{black}{
The capacity sizing based on typical single-year wind, solar, and price data may face issues of oversizing the system or RES curtailment due to the inter-annual variability of renewable energy \cite{javed2023impact}. Therefore, the impact of long-term uncertainties in renewable generation, and ammonia prices on the MSSE are analyzed. We divide recent years of wind and solar data into four scenarios based on renewable energy potential (i.e., FLH) \cite{javed2023impact}, with the FLH shown in Fig. \ref{fig:FLH}. Then, we incorporate four scenarios into the multistakeholder equilibrium framework. Since ammonia prices are negatively correlated with wind and solar resources, they are not considered separately. Furthermore, the objectives of the three stakeholders are modified as (\ref{eq:scenario}) and the solution approach is similarly applicable.
\begin{align}
	 C^k=C_{\text{inv}}^k+\sum\nolimits_{s=1}^S \pi_s C_{\text{ope}}^k(s),\ \forall k \label{eq:scenario}
\end{align}
where $s$ and $C_{\text{ope}}^k(s)$ are scenario index and its operational cost, respectively; $S$ is the number of scenarios; $\pi_s$ represents the probability of different scenario $ s$. The planning result is indicated by \textbf{F1} and presented in Table \ref{tab:comparison}. The algorithm converges after 106 iterations, with a computation time of 77.8 minutes. Similarly, as mentioned in Section \ref{sec:compute}, the time becomes shorter with more abundant computational resources. Therefore, the proposed method can effectively address uncertainty through the scenario set.}

\textcolor{black}{Comparing F1 with C1, the capacity in F1 is slightly higher to accommodate higher renewable energy output. As renewable energy potential increases, LCOA decreases, but the drop in ammonia prices leads to a higher total cost. According to the three-stakeholder profits, the uncertainty of wind and solar power does not change the dominant role (i.e., HP) in the equilibrium, but it affects the equilibrium prices. When there is a positive forecast error in renewable energy, the electricity price decreases, and RG profit shifts to HP and AS. Hydrogen price is influenced by both electricity and ammonia prices, and the profit changes of HP and AS require further consideration.}

\begin{figure}[t]
	\centering
	\includegraphics[width=3.45in]{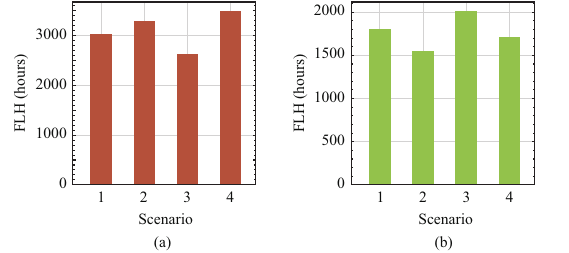}\vspace{0pt}
	\caption{\textcolor{black}{FLH of renewable power under different scenarios representing variations across multiple years. (a) Wind power. (b) Solar power.}}
	\label{fig:FLH}\vspace{-0pt}
\end{figure}

\section{Conclusions}
\label{sec:conclusions}

This study proposes a multistakeholder sizing equilibrium (MSSE) model for planning an ReP2A system, encompassing the entire process of generation, storage, and utilization of electricity, hydrogen, and ammonia.
A multicut generalized Bender decomposition (GBD) approach is developed to efficiently address long-term energy and mass balancing problems.
The following findings are drawn from the simulation results.

1) The interests of RG, HP, and AS stakeholders are not balanced due to their heterogeneous flexibility. They cannot simultaneously achieve positive profit under free competition, causing at least one stakeholder to be reluctant to invest, thereby making the ReP2A project unimplementable.

2) To increase the attractiveness and feasibility of ReP2A projects, regulators should establish a negotiation platform where stakeholders can sign benefit transfer, \textcolor{black}{ pricing, or total profit re-arrangement} agreements, enabling mutual benefits and ensuring the successful implementation of ReP2A projects.

\textcolor{black}{In the future, the following areas require further work to promote the development of ReP2A, including:}

\textcolor{black}{1) Introducing appropriate policies to guide and improve the transaction rules for electricity, hydrogen, and ammonia.}

\textcolor{black}{2) Proposing long-term planning methods that consider diverse uncertainties, such as renewable generation, energy prices, and demand.}

\textcolor{black}{3) Developing coordinated operational mechanisms for multistakeholder ReP2A systems.}

\textcolor{black}{Our future research will focus on the carbon-emission market and its impact on planning}. Additionally, the equilibrium in ammonia transactions with external chemical users within the ReP2A system will be considered.

\bibliographystyle{IEEEtran}
%\bibliography{IEEEabrv,SizOpt}

\end{document}